\newif\ifima
\newif\ifams     
\newif\ifels
\newif\ifspringer

\imafalse
\amsfalse
\elsfalse
\springerfalse

\springertrue

\ifspringer
    \RequirePackage{fix-cm}
    \documentclass[smallextended]{svjour3}
    \smartqed
\fi

\usepackage{stmaryrd,curves,amsmath,amssymb,fancybox,exscale,lineno}
\usepackage[all]{xy}
\usepackage[pdftex]{color}
\usepackage{graphicx}
\usepackage{caption,subcaption}
\usepackage{amscd}
\usepackage[colorlinks]{hyperref}

\DeclareMathOperator{\divs}{div}
\DeclareMathOperator{\divv}{Div}
\DeclareMathOperator{\curl}{curl}
\DeclareMathOperator{\Curl}{Curl}
\DeclareMathOperator{\tr}{tr}

\numberwithin{equation}{section}
\numberwithin{theorem}{section}
\numberwithin{lemma}{section}
\numberwithin{remark}{section}

\title{Recovery-based a posteriori error analysis for plate bending problems}

\def\shortTitle{Recovery-based a posteriori error analysis of plate bending problems}

\def\myAbstract{We present two new recovery-based a posteriori error estimates for the Hellan--Herrmann--Johnson method in Kirchhoff--Love plate theory. The first error estimator uses a postprocessed deflection and controls the $L^2$ moment error and the discrete $H^2$  deflection error. The second one controls the $L^2\times H^1$ total error and utilizes superconvergent postprocessed moment field and deflection.  The effectiveness of the theoretical results is numerically validated in several experiments.}

\def\myKeywords{Kirchhoff--Love plate, fourth order elliptic equation,  Hellan--Herrmann--Johnson method, a posteriori error estimates,  postprocessing,  superconvergence}

\begin{document}
\ifspringer
   \author{Yuwen Li}
  \institute{Y.~Li : Department of Mathematics, The Pennsylvania State University, University Park, PA 16802\\ {Email:} yuwenli925@gmail.com}
  \date{Received:  \  / Accepted: date}
  \maketitle
  \begin{abstract}\myAbstract\end{abstract}
  \markboth{Y.~Li}{\shortTitle}
\fi


\section{Introduction}\label{sec1}
{Let $u$ denote the deflection and $\sigma$ be the moment field of a linearly elastic thin plate, whose midsurface occupies a domain $\Omega\subset\mathbb{R}^2$. In Kirchhoff--Love plate theory, the equilibrium deflection $u$ of the plate subject to the transverse load $f\in L^2(\Omega)$ and mixed boundary conditions is described by the fourth order elliptic boundary value problem  
\begin{subequations}\label{fourthorder}
\begin{align}
\divs\divv\mathbb{M}\nabla^2 u&=f\quad \text{in } \Omega,\\
u=\partial_nu&=0\quad \text{on } \Gamma_c,\label{clamped}\\
u={\sigma}_{nn}&=0\quad \text{on }\Gamma_s,\label{simplysupported}\\
{\sigma}_{nn}=K({\sigma})&=0\quad \text{on }\Gamma_f.\label{free}
\end{align}
\end{subequations}}
Due to very high degree of $C^1$ conforming finite elements, the fourth order boundary value problem is often discretized by the nonconforming element, mixed element, or discontinuous Galerkin (dG) methods, see, e.g, \cite{Hellan1967,Herrmann1967,Morley1968,Johnson1973,CiarletRaviart1974,Miyoshi1973,BabuskaOsbornPitkaranta1980,HansboLarson2002,BrennerSung2005,WangXu2006,BehrensGuzman2011} and references therein. Among various mixed methods for plate bending, the Hellan--Herrmann--Johnson (HHJ) (cf.~\cite{Johnson1973}) method is perhaps the most famous and popular one because of  using low order polynomials and small number of degrees of freedom by hybridization (cf.~\cite{ArnoldBrezzi1985}).
The HHJ mixed method directly approximates the deflection $u$ and the moment field $\sigma$ by finite element solutions $u_h$ and $\sigma_h$, respectively.

To achieve optimal order numerical accuracy for plate bending analysis, adaptive mesh refinement based on
a posteriori error estimation is needed on domains with nonsmooth boundaries.
For the HHJ mixed method, the work \cite{HHX2011} presents a residual-type a posteriori error estimate for the moment error $\|\sigma-\sigma_h\|$, {where $\|\cdot\|$ is the $L^2$ norm}. Meanwhile, that work gives another residual error estimator for the $H^1$ deflection error $\|u-u_h\|_1$ on convex domains. The Ciarlet--Rarviart (see \cite{CiarletRaviart1974}) mixed method directly approximates $\Delta u$ and $u$ and its error estimator is given in \cite{Gudi2011}. A posteriori error estimates for dG methods in plate bending could be found in e.g., \cite{BrennerGudiSung2010,BeiraoNiiranenStenberg2010,GeorgoulisHoustonVirtanen2011,HansboLarson2011,SunHuang2018}. An error estimator for the $C^1$ element method under general boundary conditions and concentrated loads is derived in \cite{GustafssonStenbergVideman2018}. 

In the numerical literature for fourth order elliptic equations, most existing a posteriori error estimates are of residual-type, including the aforementioned ones. It is well known that recovery-based error estimators provide sharper effectiveness ratio and allow simpler implementation. 
In this work, we develop several new recovery-based a posteriori error estimates of the HHJ method based on postprocessed solutions $u_h^*$ and $\sigma_h^*$ under general boundary conditions. The construction of $u_h^*$ is in the spirit of \cite{Stenberg1991}. However, in contrast to the globally discontinuous deflection in \cite{Stenberg1991}, the new deflection $u_h^*\in C^0(\Omega)$ is conforming and is obtained by solving a well-conditioned global problem. We prove a new {quasi-optimal a priori error estimate} for $\|\sigma-\sigma_h\|+\|u-u_h^*\|_{2,h}$, where $\|\cdot\|_{2,h}$ is a discrete $H^2$ norm. Then using $u_h^*$, a simple and new a posteriori error bound $\eta_h$ is derived for controlling $\|\sigma-\sigma_h\|+\|u-u_h^*\|_{2,h}$. A similar result for mixed methods for Poisson's equation could be found in \cite{LovadinaStenberg2006}. As far as we know, all a posteriori error estimates of nonconforming and mixed methods for fourth order elliptic equations in the literature rely on the Helmholtz decomposition. In contrast, the analysis of our error estimator does not utilize Helmholtz-type decomposition. As a result, the first proposed error estimator works on multiply connected domains. In addition, this error estimator is directly applicable to the Herrmann--Miyoshi mixed method, see \eqref{HM}.

The second proposed error estimator $\zeta_h$ is designed for the lowest order HHJ method and is based on superconvergence of $|u-u_h^*|_1$ and $\|\sigma-\sigma^*_h\|$, where $\sigma_h^*=R_h\sigma_h$ is a  postprocessed $C^0$ moment field. In the literature, similar error estimators are known as superconvergent recovery-based error indicators, which are quite popular for their simplicity and asymptotic exactness, see, e.g, \cite{ZZ1992a,ZZ1992b}. The superconvergence analysis of $|u-u_h^*|_1$ is classical and works on unstructured grids. In practice, the moment variable $\sigma$ is also very important. However, there has been little work devoted to $\sigma$. An exception is  \cite{HuMa2016}, which gives a postprocessing scheme $K_h$ by edge averaging and a superconvergence estimate for $\|\sigma-K_h\sigma_h\|$ on {a special uniform mesh satisfying the assumption in Lemma \ref{superclosesigma}}. Our proposed postprocessing procedure $R_h$ solves least-squares problems on local vertex patches, see also, e.g., \cite{BankLi2019} for least-squares recovery process for Raviart--Thomas elements. We rigorously analyze the well-posedness of $R_h$, show the super-approximation of $\|\sigma-R_h\sigma\|$ under general grids, and prove superconvergence of $\|\sigma-R_h\sigma_h\|$ on structured grids. In a numerical example, we investigate a popular structured grid sequence, over which $\|\sigma-R_h\sigma_h\|$ is superconvergent while $\|\sigma-K_h\sigma_h\|$ is not.

The rest of this paper is organized as follows. In Section \ref{secModel}, we introduce basic notation for plate bending and the HHJ mixed method. In Section \ref{sec02h}, we develop a priori and a posteriori error estimates for $\|\sigma-\sigma_h\|+\|u-u_h^*\|_{2,h}$. Section \ref{sec01} is devoted to superconvergence analysis of $\|\sigma-\sigma^*_h\|+|u-u_h^*|_1$ and the corresponding recovery-based error estimator. Numerical examples including both singular and smooth problems are reported in Section \ref{secNE}.

\section{Model problem}\label{secModel}
In this section,
we first explain the notation used in the model problem \eqref{fourthorder}.
The domain $\Omega\subset\mathbb{R}^2$ has a piecewise flat boundary $\partial\Omega=\overline{\Gamma}_c\cup\overline{\Gamma}_s\cup\overline{\Gamma}_f$ with relatively open disjoint subsets $\Gamma_c, \Gamma_s, \Gamma_f$. We use  $n$ to denote the outward unit normal on $\partial\Omega$, $t$ the counterclockwise unit tangent on $\partial\Omega$. 
Let $E$ be the Young's modulus, $\nu\in[0,0.5)$ the Poisson ratio, and $d$ the thickness of the plate. Given a  symmetric $2\times2$ matrix ${\tau}$, the linear moment operator $\mathbb{M}$ is defined as 
\begin{equation*}
    \mathbb{M}{\tau}:=\frac{Ed^3}{12(1-\nu^{2})}\big((1-\nu){\tau}+\nu\tr( {\tau}){\delta}\big),
\end{equation*}
where ${\delta}$ is the $2\times2$ identity matrix, and $\tr({\tau})$ is the trace of ${\tau}$. The moment field of the plate is 
\begin{equation*}
{\sigma}:={\sigma}(u)=\mathbb{M}\nabla^2u.    
\end{equation*}
In this paper, all vectors are viewed as column vectors by default. The normal-normal and twisting components of ${\tau}$ on $\partial\Omega$ are
\begin{equation}\label{taunnnt}
    {\tau}_{nn}:={n}^\intercal{\tau}{n},\quad {\tau}_{nt}:={n}^\intercal{\tau}{t}.
\end{equation}
Let $\nabla^2$ denote the Hessian operator, {${\rm div}=\nabla\cdot$ the divergence operator for vector fields,} and $\divv$ the row-wise divergence applied to matrix-valued functions. 
By $\partial_g$ we denote the directional derivative along the unit vector $g$. The Kirchhoff shear force at the boundary $\partial\Omega$ is
\begin{equation}\label{Kirchhoff}
K({\sigma}):=(\divv{\sigma})\cdot{n}+\partial_t{\sigma}_{nt}.
\end{equation}
Here $\sigma_{nt}$ is required to be continuous at the turning points of $\Gamma_f$. In the literature, \eqref{clamped}, \eqref{simplysupported}, \eqref{free} are called clamped, simply supported, and free boundary conditions, respectively. We refer to Fig.~\ref{mixBCfigure} for an illustration. \begin{figure}[tbhp]
\centering
\includegraphics[width=6.5cm,height=5cm]{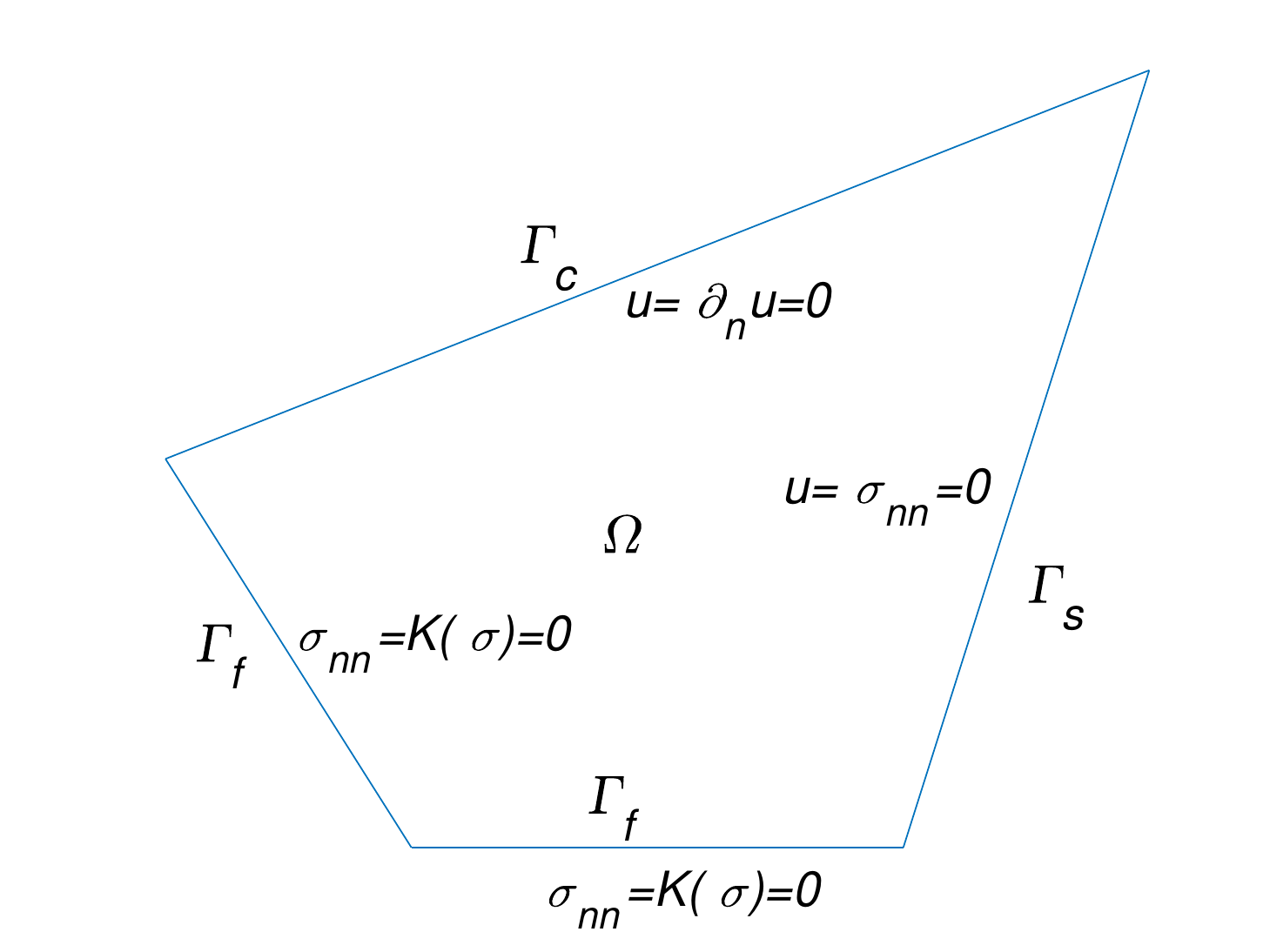}
\caption{{An example of the mixed boundary condition.}}
\label{mixBCfigure}
\end{figure}
In practice, the primal variational formulation of \eqref{fourthorder} using
\begin{equation*}
    \widetilde{U}:=\{v\in H^2(\Omega): v=0\text{ on }\Gamma_c\cup\Gamma_s,~\partial_nv=0\text{ on }\Gamma_c\}
\end{equation*}
could be discretized by conforming and nonconforming finite elements and dG methods
see, e.g., \cite{Morley1968,HansboLarson2002,BrennerSung2005,WangXu2006,GustafssonStenbergVideman2018}.

\subsection{Mixed formulation}
Let $\mathcal{T}_{h}$ be a family of shape-regular triangulation of $\Omega$. In $\mathcal{T}_h$, let $\mathcal{E}_h$,  $\mathcal{E}_{h}^{o}$, $\mathcal{E}_{h}^c$, $\mathcal{E}_{h}^f$ denote the sets of edges, interior edges, boundary edges on $\Gamma_c$, boundary edges on $\Gamma_f$, respectively. Each edge $e\in\mathcal{E}_h$ is assigned with a unit normal vector ${n}_e$, and $n_e$ is chosen to be outward when $e$ is on $\partial\Omega.$ The normal-normal component, twisting component, and $K({\tau})$ of a symmetric $2\times2$ matrix ${\tau}$ on an edge $e$ is defined in a fashion similar to \eqref{taunnnt} and \eqref{Kirchhoff} based on ${n}_e$ and ${t}_e$. 
Given a 2d subdomain or 1d submanifold $\Omega_0\subset\Omega$, let $\|\cdot\|_{m,\Omega_0}$ denote the $H^m(\Omega_0)$ Sobolev norm, $|\cdot|_{m,\Omega_0}$ the $H^m(\Omega_0)$ semi-norm, and $(\cdot,\cdot)_{\Omega_0}$ the $L^{2}(\Omega_0)$ inner product. We adopt the notation  \begin{align*}
    \|\cdot\|_{\Omega_0}=\|\cdot\|_{0,\Omega_0},\quad \|\cdot\|_{m}=\|\cdot\|_{m,\Omega},\quad|\cdot|_{m}=|\cdot|_{m,\Omega},\quad(\cdot,\cdot)=(\cdot,\cdot)_{\Omega}.
\end{align*} 
For a space $V$, we define
\begin{align*}
&[V]_{s}^{4}:=\left\{v=\begin{pmatrix}v_{11}&v_{12}\\v_{21}&v_{22}\end{pmatrix}: v_{12}=v_{21}, v_{ij}\in V,\ i,j=1,2\right\},\\
&[V]^{n}:=\{v=(v_{1},\ldots,v_{n})^\intercal: v_{i}\in V, 1\leq i\leq n\}.
\end{align*}
The fourth order equation in \eqref{fourthorder} could be recast into
\begin{subequations}\label{mix:c}
\begin{align}
\mathbb{M}^{-1}{\sigma}&=\nabla^2 u,\\
\divs\divv{{\sigma}}&=f.
\end{align}
\end{subequations}
Direct calculation shows that the inverse of $\mathbb{M}$ is 
\begin{equation*}
    \mathbb{M}^{-1}{\tau}=\frac{12}{Ed^3}\big((1+\nu){\tau}-\nu\tr( {\tau}){\delta}\big).
\end{equation*}
We shall make use of piecewise operators  $\divv_h$, $\nabla_h$, $\nabla_h^2$, i.e., for each $T\in\mathcal{T}_h,$
\begin{equation*}
    (\divv_h{\tau}_h)|_T=\divv({\tau}_h|_T),\quad(\nabla_hv_h)|_T=\nabla(v_h|_T),\quad(\nabla^2_hv_h)|_T=\nabla^2(v_h|_T).
\end{equation*}
Given $m>0$ and a fixed $s_0>0$,
we introduce the spaces
\begin{align*}
    H^m(\mathcal{T}_h):=&\{v\in L^2(\Omega): v|_T\in H^m(T)~\forall T\in\mathcal{T}_h\},\\
    U:=&\{v\in H^1(\Omega): v=0\text{ on }\Gamma_c\cup\Gamma_s\},\\
    {\Sigma}:=&\{{\tau}\in[L^2(\Omega)]_s^4: {\tau}|_T\in [H^{\frac{1}{2}+s_0}(T)]_s^4~\forall T\in\mathcal{T}_h,\\
    &\quad{\tau}_{nn}=0\text{ on }\Gamma_s\cup\Gamma_f, {\tau}_{nn}\text{ is single-valued on each }e\in\mathcal{E}_{h}^o\},
\end{align*}
and the following bilinear forms
\begin{equation}\label{bform}
\begin{aligned}
a({\sigma},{\tau}):&=(\mathbb{M}^{-1}{\sigma},{\tau}),\\
b_h({\tau},v):&=(-{\tau},\nabla^2_hv)+\langle{\tau}_{nn},\partial_n v\rangle_{\partial\mathcal{T}_h},\\
B_h({\sigma},u;{\tau},v):&=a({\sigma},{\tau})+b_h({\tau},u)+b_h({\sigma},v),
\end{aligned}
\end{equation} 
with the $L^2$ inner product on $\partial\mathcal{T}_h$
\begin{equation*}
    \langle\cdot,\cdot\rangle_{\partial\mathcal{T}_h}=\sum_{T\in\mathcal{T}_h}\langle\cdot,\cdot\rangle_{\partial T}.
\end{equation*}
For $\sigma_{nt}\in C^0(\overline{\Omega})$, $v\in H^2(\mathcal{T}_h)$, $\tau\in [H^1(\mathcal{T}_h)]_s^4$, element-wise integration-by-parts shows that
\begin{subequations}
\begin{align}
 (\divs\divv\sigma,v)&=(\sigma,\nabla_h^2v)+\langle K(\sigma),v\rangle_{\partial\mathcal{T}_h}-\langle \sigma_{nn},\partial_nv\rangle_{\partial\mathcal{T}_h},\label{ddsigma}\\
   b_h({\tau},v)&=(\divv{\tau},\nabla v)-\langle{\tau}_{nt},\partial_t v\rangle_{\partial\mathcal{T}_h}.\label{bh2}   
\end{align}    
\end{subequations}
Therefore with \eqref{mix:c}, \eqref{ddsigma}, and sufficiently regular $\sigma$, it follows that
\begin{subequations}\label{ctsHHJ}
\begin{align}
a({\sigma},{\tau})+b_h({\tau},u)&=0,\quad{\tau}\in{\Sigma},\label{ctsHHJ1}\\
b_h({\sigma},v)&=-(f,v),\quad v\in H^2(\mathcal{T}_h)\cap U.\label{ctsHHJ2}
\end{align}
\end{subequations}
Clearly $\mathbb{M}$ and $a$ are positive definite, i.e., there exist constants $m_0>0,$ $a_0>0$ relying on $E, d, \nu$ such that
\begin{equation}\label{a}
    (\mathbb{M}{\tau},{\tau})\geq m_0\|{\tau}\|_0,\quad a({\tau},{\tau})\geq a_0\|{\tau}\|_0,\quad\forall\tau\in [L^2(\Omega)]_s^4.
\end{equation}

\subsection{HHJ method}
Let $\mathcal{P}_m(\Omega_0)$ be the set of polynomials of degrees at most $m$ on $\Omega_0$.  For a fixed integer $r\geq1$, we make use of several finite element spaces
\begin{equation*}
\begin{aligned}
&S_h=S_h^r=\{v_{h}\in H^{1}(\Omega): v_{h}|_T\in\mathcal{P}_r(T),~\forall T\in\mathcal{T}_h\},\\
&U_h=U^r_h=S_h^r\cap U,\\
&\Sigma_h=\{{\tau}_{h}\in{\Sigma}: 
{\tau}_{h}|_T\in[\mathcal{P}_{r-1}(T)]_{s}^{4},~\forall T\in\mathcal{T}_h\}.
\end{aligned}
\end{equation*}
The HHJ mixed method for solving \eqref{mix:c} is to find $({\sigma}_h,u_h)
\in{\Sigma}_h\times U_h$, such that
\begin{subequations}\label{HHJ}
\begin{align}
a({\sigma}_h,{\tau}_h)+b_{h}({\tau}_{h},u_{h})&=0,\quad\forall{\tau}_h\in{\Sigma}_h,\label{HHJ1}\\
b_{h}({\sigma}_h,v_h)&=-(f,v_h),\quad\forall v_h\in U_h.\label{HHJ2}
\end{align}
\end{subequations} Subtracting \eqref{HHJ} from \eqref{ctsHHJ} leads to 
\begin{subequations}\label{error}
\begin{align}
a({\sigma}-{\sigma}_h,{\tau}_h)+b_{h}({\tau}_{h},u-u_{h})&=0,\quad\forall{\tau}_h\in{\Sigma}_h,\label{error1}\\
b_{h}({\sigma}-{\sigma}_h,v_h)&=0,\quad\forall v_h\in U_h.\label{error2}
\end{align}
\end{subequations}

The HHJ element admits a pair of commuting interpolations, {see, e.g., \cite{BabuskaOsbornPitkaranta1980,Comodi1989,BoffiBrezziFortin2013}}. The first one is the modified Lagrange interpolation $I_h=I_h^r: C^0(\overline{\Omega})\rightarrow S_h^r$ given by 
\begin{equation}\label{Ih}
    \begin{aligned}
    &(I_hv)(z)=v(z)\text{ at each vertex $z$ in }\mathcal{T}_h,\\
    &\int_e(I_hv)\phi ds=\int_ev\phi ds,\quad\forall \phi\in\mathcal{P}_{r-2}(e),~\forall e\in\mathcal{E}_h,\\
    &\int_T(I_hv)\psi dx=\int_Tv\psi dx,\quad\forall \psi\in\mathcal{P}_{r-3}(T),~\forall T\in\mathcal{T}_h.
    \end{aligned}
\end{equation}
{For the constant $s_0>0$}, the second interpolation $\Pi_h: [H^{\frac{1}{2}+s_0}(\Omega)]_{s}^{4}\cap{\Sigma}\rightarrow{\Sigma}_h$ is defined as
\begin{equation}\label{Pih}
\begin{aligned}
&\int_e(\Pi_h\tau)_{nn}\phi ds=\int_e\tau_{nn}\phi ds,
\quad\forall\phi\in\mathcal{P}_{r-1}(e),~\forall e\in\mathcal{E}_h,\\
&\int_T(\Pi_h\tau)\psi dx=\int_T\tau\psi dx,
\quad\forall\psi\in\mathcal{P}_{r-2}(T),~\forall T\in\mathcal{T}_h.
\end{aligned}
\end{equation}
It is readily checked that
\begin{subequations}
\begin{align}
    &b_h(\tau_h,I_hv)=b_h({\tau}_{h},v),\quad\forall{\tau}_h\in{\Sigma}_h,~\forall v\in C^0(\overline{\Omega}),\label{commIh}\\
    &b_h(\Pi_{h}\tau,v_h)=b_h(\tau,v_h),\quad\forall\tau\in[H^{\frac{1}{2}+s_0}(\Omega)]_{s}^{4},~ \forall v_h\in U_h.\label{commPih}
\end{align}
\end{subequations}
{Recall that $u$ is the solution of the fourth order problem \eqref{fourthorder} in the distributional sense. Therefore $u\in H^2(\Omega)$ and the nodal interpolant $I_hu$ is well-defined due to the Sobolev embedding $H^2(\Omega)\hookrightarrow C^0(\overline{\Omega})$.}

Throughout the rest of this paper, we say $c_1\lesssim c_2$ (resp.~$c_1\gtrsim c_2$) if $c_1\leq Cc_2$ (resp.~$c_1\gtrsim Cc_2$), where $C$ is a positive absolute constant relying solely on $\Omega$, $E$, $d$, $r$, and the shape-regularity of $\mathcal{T}_h$. In the analysis we may also use $C$, $C_1$, $C_2$, $\ldots$ to denote such absolute constants independent of mesh sizes. We say $c_1\simeq c_2$ provided $c_1\lesssim c_2$ and $c_2\lesssim c_1$. 

Let $h_\mathcal{T}$ and $h_{\mathcal{E}}$ be the mesh size functions such that $h_\mathcal{T}|_T=h_T:=\text{area}(T)^\frac{1}{2}$ for each $T\in\mathcal{T}_h$ and $h_{\mathcal{E}}|_e=h_e$ is the length of $e$ for each $e\in\mathcal{E}_h$.  Let $h=\max_{T\in\mathcal{T}_h}h_T$ be the maximum mesh size.
The following a priori convergence 
\begin{subequations}\label{apriori}
\begin{align}
    |u-u_h|_1&\lesssim h^r\big(|\sigma|_{r}+|u|_{r+1}\big),\label{apriori1}\\
    \|\sigma-\sigma_h\|_0&\lesssim h^r|\sigma|_r\label{apriori2}
    \end{align}
\end{subequations}
{could be found from \cite{BabuskaOsbornPitkaranta1980} and (10.3.49), (10.3.50) of \cite{BoffiBrezziFortin2013}.}

\section{A posteriori error estimation for $\|\cdot\|_0\times\|\cdot\|_{2,h}$}\label{sec02h}
Given a subset $\hat{\mathcal{E}}_h\subseteq\mathcal{E}_h$ , the $L^2$ norm on $\hat{\mathcal{E}}_h$ is 
\begin{equation*}
    \|\cdot\|_{\hat{\mathcal{E}}_h}:=\big(\sum_{e\in\hat{\mathcal{E}}_h}\|\cdot\|^2_e\big)^\frac{1}{2}.
\end{equation*}
Let $\Gamma^c_h:=\mathcal{E}_h^o\cup\mathcal{E}_h^c$ and $\Gamma^f_h:=\mathcal{E}_h^o\cup\mathcal{E}_h^f$.
On an interior edge $e\ni x,$ let $\llbracket\omega\rrbracket(x):=\lim_{s\to0^+}\big[\omega(x+sn_e)-\omega(x-sn_e)\big]$ denote the jump function of $\omega$ across $e$. On an boundary edge $e\subset\partial\Omega,$ $\llbracket\omega\rrbracket$ is the restriction of $\omega$ on $e.$
Following \cite{BabuskaOsbornPitkaranta1980}, we define the mesh-dependent $H^2$-norms
\begin{align*}
    \|v_h\|_{2,h}^2&=\|\nabla^2_hv_h\|^2+\|h_\mathcal{E}^{-\frac{1}{2}}\llbracket\partial_nv_h\rrbracket\|^2_{\Gamma^c_h},\\
    \|v_h\|_{2,h,T}^2&=\|\nabla^2v_h\|_T^2+\|h_\mathcal{E}^{-\frac{1}{2}}\llbracket\partial_nv_h\rrbracket\|^2_{\partial T\backslash(\Gamma_f\cup\Gamma_s)},\quad T\in\mathcal{T}_h.
\end{align*}
In this section, we present an error estimator for controlling the $\|\cdot\|_{0}$-norm of the moment error and $\|\cdot\|_{2,h}$-norm of the deflection error. The original solution $u_h$ does not optimally converge to $u$ in the $\|\cdot\|_{2,h}$-norm. For example, $\|u-u_h\|_{2,h}$ is not convergent at all in the lowest order case $r=1$. To remedy the situation, we reconstruct a more accurate deflection $u_h^*$ by postprocessing $u_h$ and then derive a priori and a posteriori error estimates for $\|\sigma-\sigma_h\|_0+\|u-u_h^*\|_{2,h}.$

Let $I$ be the identity mapping and $U_h^*=U_h^{r+1}$. We utilize the space of bubble functions
\begin{equation*}
W_h:=(I-I^r_h)U_h^{r+1}.    
\end{equation*}
For example, when $r=1,$ $W_h$ is spanned by edge bubbles (a function of unit size locally supported on two elements sharing an edge). 
The postprocessed deflection is given by $$u_h^*:=u_h+w_h\in U_h^{r+1},$$ where the high frequency component $w_h\in W_h$ solves the global problem
\begin{equation}\label{wh}
(\mathbb{M}\nabla_h^2w_h,\nabla_h^2v_h)=(\sigma_h-\mathbb{M}\nabla_h^2u_h,\nabla_h^2v_h),\quad\forall v_h\in W_h.
\end{equation}
Clearly \eqref{wh} is well defined.
By the construction, we obtain the orthogonality 
\begin{equation}\label{L2}
    (\mathbb{M}\nabla_h^2u^*_h,\nabla_h^2v_h)=(\sigma_h,\nabla_h^2v_h),\quad v_h\in W_h.
\end{equation}
The reconstruction of $u_h^*$ is slightly different from the element-wise postprocessing in \cite{Stenberg1991}. In particular, it is noted that $u_h^*$ is $C^0$ continuous while the postprocessed deflection in \cite{Stenberg1991} is broken and completely discontinuous. Similar discontinuous postprocessing deflections in dG methods could be found in e.g., \cite{SunHuang2018}. Such  discontinuity of deflections is {somehow} undesirable in engineering analysis, {e.g., when visualizing physical quantities dependent on discontinuous deflections.} In theory, the  $C^0$ continuity of $u_h^*$ will greatly facilitate subsequent a priori and a posteriori error analysis.
\begin{remark}
Although \eqref{wh} is global, the basis of $W_h$ consists of edge and volume bubble functions. It is well known that the stiffness matrix of \eqref{wh} under such a basis is spectrally equivalent to its diagonal. As a consequence, \eqref{wh} could be optimally solved by the diagonally preconditioned conjugate gradient method within uniformly bounded number of iterations independent of the mesh size. Therefore the computational cost of $u_h^*$ is comparable to the local postprocessing scheme in \cite{Stenberg1991}. The hierarchical decomposition $U_h^*=U_h\oplus W_h$ and similar well-conditioned global problems are also used in the hierarchical basis error estimator, see \cite{BankSmith1993,Bank1996}. {Postprocessing for numerical solutions of elliptic equations by global Ritz projection could also be found in, e.g., \cite{DednerGiesselmannPryerRyan2020}.}
\end{remark}

A key tool in our analysis is the following approximation result, which is a special case of Lemma 3.1 in \cite{GeorgoulisHoustonVirtanen2011}. For any $H^1$ conforming $v_h\in U^r_h,$
there exists $\phi\in \widetilde{U}$ which is a $C^1$ piecewise polynomial satisfying
\begin{equation}\label{approx2h}
    \|h_\mathcal{T}^{-2}(v_h-\phi)\|+\|h_\mathcal{T}^{-1}\nabla_h(v_h-\phi)\|+\|\nabla_h^2(v_h-\phi)\|\leq C\|h_\mathcal{E}^{-\frac{1}{2}}\llbracket\partial_n v_h\rrbracket\|_{\Gamma^c_h}.
\end{equation}
The following lemma is a direct consequence of \eqref{approx2h}.
\begin{lemma}\label{approxIh}
For any $v_h\in U_h^r$,
we have
\begin{equation*}
    \|h_\mathcal{T}^{-2}(v_h-I_hv_h)\|\lesssim \|v_h\|_{2,h}.
\end{equation*}
\end{lemma}
\begin{proof}
For $\psi\in H^2(T)$ and $T\in\mathcal{T}_h,$ the stability of the Lagrange interpolation $I_h$ implies
\begin{equation}\label{Ihstable}
    \|I_h\psi\|_T\lesssim \|\psi\|_T+h_T|\psi|_{1,T}+h_T^2|\psi|_{2,T}.
\end{equation}
Let $\phi$ be given in \eqref{approx2h}. Using the triangle inequality, \eqref{Ihstable}, {and the standard interpolation error estimate for $I_h\phi$,} we have 
\begin{equation*}
\begin{aligned}
    &\|h_\mathcal{T}^{-2}(v_h-I_hv_h)\|\leq\|h_\mathcal{T}^{-2}(v_h-\phi)\|+\|h_\mathcal{T}^{-2}(\phi-I_h\phi)\|+\|h_\mathcal{T}^{-2}I_h(\phi-v_h)\|\\
    &\leq\|h_\mathcal{T}^{-2}(v_h-\phi)\|+C\|\nabla^2\phi\|+\|h_\mathcal{T}^{-1}\nabla(v_h-\phi)\|+\|\nabla^2_h(v_h-\phi)\|\\
    &\leq\|h_\mathcal{T}^{-2}(v_h-\phi)\|+C\|\nabla_h^2(\phi-v_h)\|+C\|\nabla_h^2v_h\|+\|h_\mathcal{T}^{-1}\nabla(v_h-\phi)\|.
\end{aligned}
\end{equation*}
Combining the above inequality with \eqref{approx2h} and \eqref{Ihstable} completes the proof.
\qed\end{proof}

In the sequel, we need the trace inequality
\begin{equation}\label{trace}
    \|v\|_{\partial T}\lesssim h_T^{-\frac{1}{2}}\|v\|_T+h_T^{\frac{1}{2}}\|\nabla v\|_{T},\quad\forall v\in H^1(T).
\end{equation}

\subsection{A priori error estimation under $\|\cdot\|_0\times\|\cdot\|_{2,h}$} 
It is shown in \cite{HHX2011} that 
\begin{equation}\label{infsup0}
    \sup_{0\neq\tau_h\in \Sigma_h}\frac{b_h(\tau_h,v_h)}{\|{\tau}_h\|}\gtrsim\|v_h\|_{2,h},\quad\forall v_h\in U_h.
\end{equation}
A combination of  \eqref{infsup0}, the positive definiteness of $a$, and Babu\v{s}ka--Brezzi theory then yields
\begin{equation}\label{infsup1}
    \inf_{0\neq(\xi_h,w_h)\in\Sigma_h\times U_h}\sup_{0\neq(\tau_h,v_h)\in{\Sigma}_h\times U_h}\frac{B_h(\xi_h,w_h;\tau_h,v_h)}{\big(\|\xi_h\|+\|w_h\|_{2,h}\big)\big(\|\tau_h\|+\|v_h\|_{2,h}\big)}\gtrsim1,
\end{equation}
see \cite{Babuska1973,Brezzi1974,XuZikatanov2003}.
We introduce the modified bilinear form
\begin{equation*}
    \widetilde{B}_h(\xi_h,w^*_h;\tau_h,v^*_h)=B_h(\xi_h,w^*_h;\tau_h,v^*_h)+(\mathbb{M}\nabla_h^2w^*_h-\xi_h,\nabla_h^2(I-I_h)v^*_h),
\end{equation*}
see \cite{LovadinaStenberg2006} for a similar bilinear form of mixed methods for Poisson's equation.
The next lemma presents an inf-sup condition of $\widetilde{B}_h$.
\begin{lemma}\label{infsup2}
It holds that
\begin{equation*}
    \inf_{0\neq({\xi}_h,w_h^*)\in\Sigma_h\times U_h^*}\sup_{0\neq(\tau_h,v_h^*)\in\Sigma_h\times U_h^*}\frac{\widetilde{B}_h(\xi_h,w_h^*;\tau_h,v_h^*)}{\big(\|\xi_h\|+\|w_h^*\|_{2,h}\big)\big(\|{\tau}_h\|+\|v_h^*\|_{2,h}\big)}\gtrsim1.
\end{equation*}
\end{lemma}
\begin{proof}
Given $(\xi_h,w_h^*)\in\Sigma_h\times U_h^*$, \eqref{infsup1} implies that there exist $(\tau_h,v_h)\in\Sigma_h\times U_h$ such that
\begin{subequations}\label{Bhterm}
    \begin{align}
    \|\tau_h\|+\|v_h\|_{2,h}&\lesssim1,\label{Bhterm1}\\
B_h(\xi_h,I_hw_h^*;\tau_h,v_h)&\gtrsim\|\xi_h\|+\|I_hw_h^*\|_{2,h}.\label{Bhterm2}
    \end{align}
\end{subequations}
By  \eqref{commIh}, \eqref{Bhterm2}, we obtain
\begin{equation}\label{Btilde1}
    \begin{aligned}
&\widetilde{B}_h(\xi_h,w_h^*;\tau_h,v_h)=a(\xi_h,\tau_h)+b_h(\tau_h,w^*_h)+b_h(\xi_h,v_h)\\
&=a(\xi_h,\tau_h)+b_h(\tau_h,I_hw^*_h)+b_h(\xi_h,v_h)\gtrsim\|\xi_h\|+\|I_hw_h^*\|_{2,h}.
    \end{aligned}
\end{equation}
Let $\tilde{v}^*_h=\frac{(I-I_h)w^*_h}{\|\nabla^2_h(I-I_h)w^*_h\|}$. We use \eqref{trace} and the Cauchy--Schwarz inequality to obtain
\begin{equation}\label{Btilde2}
    \begin{aligned}
&\widetilde{B}_h(\xi_h,w_h^*;0,\tilde{v}^*_h)=b_h(\xi_h,\tilde{v}^*_h)+(\mathbb{M}\nabla_h^2w^*_h-\xi_h,\nabla_h^2\tilde{v}^*_h)\\
&=-2(\xi_h,\nabla_h^2\tilde{v}^*_h)+\langle(\xi_h)_{nn},\partial_n \tilde{v}_h^*\rangle_{\partial\mathcal{T}_h}\\
&\quad+(\mathbb{M}\nabla_h^2(I-I_h)w^*_h,\nabla_h^2\tilde{v}^*_h)+(\mathbb{M}\nabla_h^2I_hw^*_h,\nabla_h^2\tilde{v}^*_h)\\
&\geq -2\|\xi_h\|-C^\frac{1}{2}_1\|\xi_h\|^\frac{1}{2}\|h_\mathcal{E}^{-\frac{1}{2}}\llbracket{\partial_n\tilde{v}_h^*}\rrbracket\|^\frac{1}{2}_{\Gamma^c_h}\\
&\quad+C_2\|\nabla_h^2(I-I_h)w^*_h\|-C_3\|\nabla^2_hI_hw_h^*\|.
    \end{aligned}
\end{equation}
For each $T\in\mathcal{T}_h,$ we note that $(I-I_h)w^*_h\in W_h$ vanishes at each vertex of $T.$ As a result, a local  scaling argument leads to
\begin{equation}\label{localIIhwhstar}
    \|\nabla(I-I_h)w^*_h\|_T\lesssim h_T\|\nabla^2(I-I_h)w^*_h\|_T,\quad\forall T\in\mathcal{T}_h.
\end{equation}
It follows from  \eqref{localIIhwhstar} and \eqref{trace} that
\begin{equation}\label{edgeIIhwhstar}
    \|h_\mathcal{E}^{-\frac{1}{2}}\llbracket \partial_n(I-I_h)w^*_h\rrbracket\|_{\Gamma^c_h}\leq C_4\|\nabla_h^2(I-I_h)w^*_h\|.
\end{equation}
Therefore combining \eqref{Btilde2} with \eqref{edgeIIhwhstar} and using a mean value inequality, we obtain
\begin{equation}\label{Btilde3}
    \begin{aligned}
&\widetilde{B}_h(\xi_h,w_h^*;0,\tilde{v}^*_h)\geq-(2+\varepsilon^{-1})\|\xi_h\|\\
&\quad+\left(C_2-\frac{\varepsilon}{4}C_1C_4\right)\|\nabla_h^2(I-I_h)w^*_h\|-C_3\|\nabla^2_hI_hw_h^*\|,
    \end{aligned}
\end{equation}
where $\varepsilon>0$ is set to be $\varepsilon=\frac{2C_2}{C_1C_4}$.
Let $v_h^*=v_h+t\tilde{v}_h^*$. Using \eqref{Btilde1}, \eqref{Btilde3}, a sufficiently small $t>0$, \eqref{edgeIIhwhstar}, and a triangle inequality, we have
\begin{equation}\label{Btilde4}
\begin{aligned}
\widetilde{B}_h(\xi_h,w_h^*;\tau_h,v^*_h)&\gtrsim\|\xi_h\|+\|I_hw_h^*\|_{2,h}+\|\nabla_h^2(I-I_h)w^*_h\|\\
&\gtrsim\|\xi_h\|+\|w_h^*\|_{2,h}.
\end{aligned}
\end{equation}
On the other hand, combining  \eqref{Bhterm1}, the definition of ${v}^*_h$, and \eqref{edgeIIhwhstar}, we obtain
\begin{equation*}
 \|\tau_h\|+\|v^*_h\|_{2,h}\lesssim1
\end{equation*}
and complete the proof.
\qed\end{proof}

Motivated by Lemma \ref{infsup2}, we are able to derive a new quasi-optimal a priori error estimate under $\|\cdot\|_0\times\|\cdot\|_{2,h}$. In the following, let  $Q_h^r$ denote the $L^2$ projection onto the space of globally discontinuous and piecewise polynomials of degree at most $r$ on $\mathcal{T}_h$, and $Q_h^{-2}=Q_h^{-1}=0$.
\begin{theorem}\label{quasioptimal}
It holds that
\begin{equation*}
    \|{\sigma}-{\sigma}_h\|+\|u-u_h^*\|_{2,h}\lesssim\inf_{\tau_h\in\Sigma_h, v_h^*\in U_h^*}\big(\|\sigma-\tau_h\|+\|u-v_h^*\|_{2,h}\big)+\|h_\mathcal{T}^2(f-Q_h^{r-3}f)\|.
\end{equation*}
\end{theorem}
\begin{proof}
For any ${\tau}_h\in{\Sigma}_h$, $v_h^*\in U_h^*$, direct calculation shows that
\begin{equation}\label{Btilde0}
    \begin{aligned}
    &\widetilde{B}_h(\sigma-\sigma_h,u-u^*_h;\tau_h,v^*_h)=a(\sigma-\sigma_h,\tau_h)+b_h(\tau_h,u-u^*_h)\\
    &+b_h(\sigma-\sigma_h,v^*_h)+(\mathbb{M}\nabla_h^2(u-u^*_h)-\sigma+\sigma_h,\nabla_h^2(I-I_h)v^*_h).
    \end{aligned}
\end{equation}
Using \eqref{Btilde0}, \eqref{commIh}, \eqref{HHJ}, \eqref{ctsHHJ}, \eqref{L2},  $I_hu_h^*=u_h$, we obtain
\begin{equation}\label{inconsistency}
    \begin{aligned}
    &\widetilde{B}_h(\sigma-\sigma_h,u-u^*_h;{\tau}_h,v^*_h)=a({\sigma}-{\sigma}_h,{\tau}_h)+b_h({\tau}_h,I_hu-I_hu^*_h)\\
    &\qquad+b_h(\sigma-\sigma_h,v^*_h)+(\sigma_h-\mathbb{M}\nabla_h^2u^*_h,\nabla_h^2(I-I_h)v^*_h)\\
    &\quad=a(\sigma-\sigma_h,\tau_h)+b_h(\tau_h,u-u_h)-(f,v_h^*)-b_h(\sigma_h,I_hv^*_h)\\
    &\quad=-(f,v_h^*-I_hv^*_h)=-(f-Q_h^{r-3}f,v_h^*-I_hv^*_h),
    \end{aligned}
\end{equation}
{which is an error term due to inconsistency of $\widetilde{B}_h$.
It then follows from a  Strang's lemma (cf.~\cite{BrennerScott2008}) for nonconforming methods, the inf-sup condition for $\widetilde{B}_h$ in Lemma \ref{infsup2},}  and \eqref{inconsistency} that
\begin{equation}\label{last}
\begin{aligned}
&\|\sigma-\sigma_h\|+\|u-u_h^*\|_{2,h}\\
&\lesssim\inf_{\tau_h\in\Sigma_h, v_h^*\in U_h^*}\big(\|\sigma-\tau_h\|+\|u-v_h^*\|_{2,h}\big)\\
&\quad+\sup_{\substack{\tau_h\in\Sigma_h, v_h^*\in U_h^*\\\|\tau_h\|+\|v_h^*\|_{2,h}=1}}\widetilde{B}_h(\sigma-\sigma_h,u-u^*_h;\tau_h,v^*_h)\\
&\lesssim\inf_{\tau_h\in\Sigma_h, v_h^*\in {U_h^*}}\big(\|\sigma-\tau_h\|+\|u-v_h^*\|_{2,h}\big)\\
&\quad+\sup_{v_h^*\in U_h^*, \|v_h^*\|_{2,h}=1}(f-Q_h^{r-3}f,v_h^*-I_hv^*_h).
\end{aligned}
\end{equation}
Applying Lemma \ref{approxIh} to the last term in \eqref{last} completes
the proof. 
\qed\end{proof}

For boundary value problems with sufficiently smooth solution $(\sigma,u)$, the quasi-optimal error estimate in Theorem \ref{quasioptimal} implies the optimal order rate of convergence
\begin{align*}
    \|\sigma-\sigma_h\|+\|u-u_h^*\|_{2,h}\lesssim h^r.
\end{align*}
In the literature, the HHJ mixed method with some element-wise postprocessed deflection in e.g., \cite{Stenberg1991} also fulfills the same optimal order convergence under $\|\cdot\|_{0}\times\|\cdot\|_{2,h}$ for smooth problems.
However, Theorem \ref{quasioptimal} is stronger than the aforementioned a priori error estimates because it provides best approximation in general situations, regardless of the solution regularity.

\subsection{A posteriori error estimate by $u_h^*$}
Based on the reconstructed deflection $u_h^*$, we give a new a posteriori error bound $\eta_h=\big(\sum_{T\in\mathcal{T}_h}\eta_h(T)^2\big)^\frac{1}{2}$ for controlling $\|\sigma-\sigma_h\|+\|u-u_h^*\|_{2,h}$  with the following error indicator
\begin{equation*}
    \eta_h(T):=\big(\|\mathbb{M}^{-1}\sigma_h-\nabla_h^2u_h^*\|^2_T+h_T^4\|f-\divs\divv\sigma_h\|^2_T+\sum_{\substack{e\subset\partial T,\\ e\in\Gamma^c_h}}h_e^{-1}\|\llbracket\partial_nu_h^*\rrbracket\|^2_e\big)^\frac{1}{2}.
\end{equation*}
\begin{theorem}\label{upperbound}
We have
\begin{equation*}
    \|\sigma-\sigma_h\|+\|u-u_h^*\|_{2,h}\lesssim\eta_h.
\end{equation*}
\end{theorem}
\begin{proof}
Let $\phi\in\widetilde{U}$ be given in \eqref{approx2h} such that
\begin{equation}\label{phiu}
    \|\nabla_h^2(u_h^*-\phi)\|\leq C^\frac{1}{2}_1\|h_\mathcal{E}^{-\frac{1}{2}}\llbracket\partial_n u_h^*\rrbracket\|_{\Gamma^c_h}.
\end{equation}
Let  $v=u-\phi\in\widetilde{U}$. We then proceed with the following splitting
\begin{equation}\label{sigmasigmah}
\begin{aligned}
&C_2\|\sigma-\sigma_h\|^2\leq a(\sigma-{\sigma}_h,{\sigma}-{\sigma}_h)\\
&=({\sigma}-{\sigma}_h,\nabla_h^2u_h^*-\mathbb{M}^{-1}{\sigma}_h)+({\sigma}-{\sigma}_h,\nabla^2v)+({\sigma}-{\sigma}_h,\nabla_h^2(\phi-u_h^*)).
\end{aligned}
\end{equation}
The mean value inequality with $\varepsilon>0$ and \eqref{phiu} yield
\begin{equation}\label{1st3rdterm}
    \begin{aligned}
    &|({\sigma}-{\sigma}_h,\nabla_h^2u_h^*-\mathbb{M}^{-1}{\sigma}_h)|\leq\frac{\varepsilon}{2}\|{\sigma}-{\sigma}_h\|^2+\frac{\varepsilon^{-1}}{2}\|\nabla_h^2u_h^*-\mathbb{M}^{-1}{\sigma}_h\|^2,\\
    &|({\sigma}-{\sigma}_h,\nabla_h^2(\phi-u_h^*))|\leq\frac{\varepsilon}{2}\|{\sigma}-{\sigma}_h\|^2+\frac{\varepsilon^{-1}}{2}C_1\|h_\mathcal{E}^{-\frac{1}{2}}\llbracket\partial_n u_h^*\rrbracket\|^2_{\Gamma^c_h}.
    \end{aligned}
\end{equation}
Integrating by parts leads to 
\begin{equation}\label{2ndterm}
    \begin{aligned}
&({\sigma}-{\sigma}_h,\nabla^2v)=-(\divv({\sigma}-{\sigma}_h),\nabla v)+\langle({\sigma}-{\sigma}_h)n,\nabla v\rangle_{\partial\mathcal{T}_h}\\
    &=-(\divv(\sigma-{\sigma}_h),\nabla v)+\langle({\sigma}-\sigma_h)_{nt},\partial_tv\rangle_{\partial\mathcal{T}_h}=-b_h({\sigma}-\sigma_h,v),
    \end{aligned}
\end{equation}
where $\langle({\sigma}-{\sigma}_h)_{nn},\partial_n v\rangle_{\partial\mathcal{T}_h}=0$ is used in the last equality. The interpolant  $v_h=I_hv\in U_h$ satisfies
\begin{equation}\label{vvh}
    \|h_\mathcal{T}^{-2}(v-v_h)\|+\|h^{-\frac{3}{2}}_\mathcal{E}(v-v_h)\|\lesssim|v|_2.
\end{equation}
It follows from \eqref{error2}, \eqref{ctsHHJ2}, \eqref{HHJ2} and integration by parts on each edge that
\begin{equation}\label{2ndtermfurther}
    \begin{aligned}
    &-b_h({\sigma}-{\sigma}_h,v)=-b_h({\sigma}-{\sigma}_h,v-v_h)=(f,v-v_h)+b_h({\sigma}_h,v-v_h)\\
    &\quad=(f-\divs\divv{\sigma}_h,v-v_h)+\langle(\divv{\sigma}_h){n},v-v_h\rangle_{\partial\mathcal{T}_h}\\
    &\qquad-\langle({\sigma}_h)_{nt},\partial_t(v-v_h)\rangle_{\partial\mathcal{T}_h}\\
    &\quad=(f-\divs\divv{\sigma}_h,v-v_h)+\langle \llbracket K({\sigma}_h)\rrbracket,v-v_h\rangle_{\Gamma^f_h},
    \end{aligned}
\end{equation}
In the last equality, we use $v-v_h=0$ at each vertices of $\partial T$ and on $\partial\Omega\backslash\Gamma_f$.
It is proved in Lemma 3.3 of \cite{HHX2011} that
\begin{equation}\label{dominance}
    \|h_\mathcal{E}^\frac{3}{2}\llbracket K(\sigma_h)\rrbracket\|_{\Gamma_h^f}\lesssim\|h_\mathcal{T}^{2}(f-\divs\divv\sigma_h)\|.
\end{equation}
Therefore combining \eqref{2ndterm}--\eqref{dominance} and using the Cauchy--Schwarz and triangle inequalities, we have
\begin{equation}\label{2ndtermfurthermore}
\begin{aligned}
|({\sigma}-{\sigma}_h,\nabla^2v)|&\leq C^\frac{1}{2}_3\|h_\mathcal{T}^{2}(f-\divs\divv{\sigma}_h)\|\|\nabla^2v\|\\
&\leq \varepsilon^{-1}C_3\|h_\mathcal{T}^{2}(f-\divs\divv{\sigma}_h)\|^2\\
&+\frac{\varepsilon}{2}\|\nabla_h^2(u-u_h^*)\|^2+\frac{\varepsilon}{2}\|\nabla_h^2(u_h^*-\phi)\|^2.
\end{aligned}
\end{equation}
It follows from \eqref{sigmasigmah}, \eqref{1st3rdterm}, \eqref{2ndtermfurthermore}, \eqref{phiu} and the triangle inequality that
\begin{equation}\label{sigmapart}
    \begin{aligned}
&C_2\|{\sigma}-{\sigma}_h\|^2\leq\varepsilon\|{\sigma}-{\sigma}_h\|^2+\frac{\varepsilon^{-1}}{2}\|\nabla_h^2u_h^*-\mathbb{M}^{-1}{\sigma}_h\|^2\\
&\quad+\varepsilon^{-1}C_3\|h_\mathcal{T}^{2}(f-\divs\divv{\sigma}_h)\|^2+\frac{\varepsilon}{2}\|\nabla_h^2(u-u_h^*)\|^2\\
&\quad+\left(\frac{\varepsilon^{-1}}{2}+\frac{\varepsilon}{2}\right)C_1\|h_\mathcal{E}^{-\frac{1}{2}}\llbracket\partial_n u_h^*\rrbracket\|^2_{\Gamma^c_h}.
\end{aligned}
\end{equation}
Meanwhile we have
\begin{subequations}\label{upart}
\begin{align}
&\|h_e^{-\frac{1}{2}}\llbracket\partial_n (u-u_h^*)\rrbracket\|_{\Gamma^c_h}=\|h_e^{-\frac{1}{2}}\llbracket\partial_n u_h^*\rrbracket\|_{\Gamma^c_h},\label{normaljump}\\
&\|\nabla_h^2(u-u_h^*)\|^2\leq2\|\mathbb{M}^{-1}({\sigma}-{\sigma}_h)\|^2+2\|\mathbb{M}^{-1}{\sigma}_h-\nabla_h^2u_h^*\|^2\\
&\hspace{2.3cm}\leq2C_4\|{\sigma}-{\sigma}_h\|^2+2\|\mathbb{M}^{-1}{\sigma}_h-\nabla_h^2u_h^*\|^2.\nonumber
\end{align}
\end{subequations}
Combining \eqref{sigmapart} with {$\varepsilon=\min(\frac{C_2}{4},\frac{C_2}{4C_4})$} and \eqref{upart} completes the proof.
\qed\end{proof}
\begin{remark}
Based on the Helmholtz decomposition in \cite{BeiraoNiiranenStenberg2007} for symmetric tensors, the work \cite{HHX2011} derived residual-type a posteriori error estimates for $\|\sigma-\sigma_h\|$ on polygonal domains.  Without using the Helmholtz decomposition, our analysis in Theorem \ref{upperbound} works on more general domain $\Omega$, e.g., domains with holes. 
\end{remark}

{Without absorbing the edge term involving $K(\sigma_h)$ in \eqref{dominance}, the error estimator $\bar{\eta}_h=\big(\sum_{T\in\mathcal{T}_h}\bar{\eta}_h(T)^2\big)^\frac{1}{2}$ with
\begin{align*}
    \bar{\eta}_h(T)&:=\big(\|\mathbb{M}^{-1}\sigma_h-\nabla_h^2u_h^*\|^2_T+h_T^4\|f-\divs\divv\sigma_h\|^2_T\\
    &\quad+\sum_{\substack{e\subset\partial T,\\ e\in\Gamma^c_h}}h_e^{-1}\|\llbracket\partial_nu_h^*\rrbracket\|^2_e+\sum_{\substack{e\subset\partial T,\\ e\in\Gamma^f_h}}h_e^3\|\llbracket K(\sigma_h)\rrbracket\|^2_e\big)^\frac{1}{2}
\end{align*}
provides an alternative upper bound for $\|\sigma-\sigma_h\|+\|u-u_h^*\|_{2,h}$ up to a possibly smaller multiplicative constant. For the lowest order HHJ method, $K(\sigma_h)=0$ and thus $\eta_h$, $\bar{\eta}_h$ coincide.}
The efficiency analysis of $\eta_h(T)$ is  straightforward.
\begin{theorem}
For each $T\in\mathcal{T}_h$ we have
\begin{equation*}
\eta_h(T)\lesssim\|\sigma-\sigma_h\|_T+\|u-u_h^*\|_{2,h,T}+h_T^2\|f-Q^r_hf\|_T.
\end{equation*}
\end{theorem}
\begin{proof}
The standard bubble function technique (cf.~\cite{Verfurth2013}) yields
\begin{equation*}
    h^2_T\|f-\divs\divv\sigma_h\|_T\lesssim\|\sigma-\sigma_h\|_T+h_T^2\|f-Q^r_hf\|_T.
\end{equation*}
Meanwhile, the triangle inequality implies
\begin{equation*}
\|\mathbb{M}^{-1}\sigma_h-\nabla_h^2u_h^*\|_T\leq C\|\sigma-\sigma_h\|_T+\|\nabla^2u-\nabla_h^2u_h^*\|_T.
\end{equation*}
Collecting the above two inequalities with \eqref{normaljump} completes the proof.
\qed\end{proof}

The mixed formulation \eqref{ctsHHJ} could also be discretized by the Herrmann--Miyoshi method (cf.~\cite{BabuskaOsbornPitkaranta1980,Herrmann1967,Miyoshi1973}): Find $(\hat{\sigma}_h,\hat{u}_h)\in[S^r_h]_s^4\times U^r_h$ such that
\begin{subequations}\label{HM}
\begin{align}
a(\hat{\sigma}_h,\tau_h)+b_h(\tau_h,\hat{u}_h)&=0,\quad\forall\tau_h\in[S^r_h]_s^4,\\
b_{h}(\hat{\sigma}_h,v_h)&=-(f,v_h),\quad\forall v_h\in U^r_h.
\end{align}
\end{subequations}
The difference between \eqref{HM} and the HHJ method \eqref{HHJ} is the use of the equal-order nodal moment space  $[S^r_h]^4_s$ in \eqref{HM}. 
When $r\geq2,$ the a priori convergence of $\|\sigma-\hat{\sigma}_h\|+\|u-\hat{u}_h\|_{2,h}$ is shown in \cite{BabuskaOsbornPitkaranta1980}.

The a posteriori analysis in Theorem \ref{upperbound} relies solely on the $H^1$ conformity of the numerical deflection $u_h^*$. Since $\hat{u}_h$ is also $C^0$ continuous, the a posteriori error analysis in this section could be directly applied to  \eqref{HM}. Let
\begin{equation*}
    \hat{\eta}_h(T)=\big(\|\mathbb{M}^{-1}\hat{\sigma}_h-\nabla_h^2\hat{u}_h\|^2_T+h_T^4\|f-\divs\divv\hat{\sigma}_h\|^2_T+\sum_{\substack{e\subset\partial T,\\ e\in\Gamma^c_h}}h_e^{-1}\|\llbracket\partial_n\hat{u}_h\rrbracket\|^2_e\big)^\frac{1}{2}.
\end{equation*}
The a posteriori error estimate for the Herrmann--Miyoshi method \eqref{HM} reads
\begin{equation*}
    \big(\sum_{T\in\mathcal{T}_h}\hat{\eta}_h(T)^2\big)^\frac{1}{2}-\|h_\mathcal{T}^2(f-Q^r_hf)\|\lesssim\|\sigma-\hat{\sigma}_h\|+\|u-\hat{u}_h\|_{2,h}\lesssim\big(\sum_{T\in\mathcal{T}_h}\hat{\eta}_h(T)^2\big)^\frac{1}{2}.
\end{equation*}
\section{A posteriori error estimate for $\|\cdot\|_0\times\|\cdot\|_1$}\label{sec01}
In this section, we present superconvergence results and the induced recovery-based error estimator with respect to the norm $\|\cdot\|\times\|\cdot\|_1$. The theoretical foundation of superconvergent recovery error estimators hinges on the analytical solution regularity and the mesh structure, which is unrealistic in adaptive methods. However, in practice, such recovery-type error estimators are often reliable, efficient, and even asymptotically exact for singular problems under graded meshes.

In particular, for the lowest order HHJ method ($r=1$), we present and theoretically validate a new postprocessing procedure $R_h$ based on the superconvergent patch recovery. Under common assumptions, the postprocessed moment $\sigma^*_h=R_h\sigma_h$ is shown to be superconvergent to $\sigma_h$, i.e.,  $\|\sigma-\sigma^*_h\|=O(h^{1+\rho})$ is a higher order term with some $\rho>0$. It then follows from the triangle inequality
\begin{equation*}
    \|\sigma_h-\sigma^*_h\|-\|\sigma-\sigma^*_h\|\leq\|\sigma-\sigma_h\|\leq\|\sigma_h-\sigma^*_h\|+\|\sigma-\sigma^*_h\|
\end{equation*}
that $\|\sigma_h-\sigma^*_h\|$ is an asymptotically exact error estimator for the moment error, i.e., the effectiveness index $\frac{\|\sigma-\sigma_h\|}{\|\sigma_h-\sigma^*_h\|}$ goes to 1 as $h\rightarrow0.$ For the same reason,  combining it with $u_h^*$ leads to the asymptotically exact a posteriori bound $\zeta_h=\big(\|\sigma_h-\sigma^*_h\|^2+|u_h-u_h^*|^2_1\big)^\frac{1}{2}$ for the total error under the norm $\|\cdot\|\times|\cdot|_1$, i.e., $\big(\|\sigma-\sigma_h\|^2+|u-u_h|^2_1\big)^\frac{1}{2}/\zeta_h$ approaches 1 as $h\rightarrow0.$

\subsection{Superconvergence of the deflection} 
It is not hard to see that {the inf-sup condition \eqref{infsup0}} yields an improved error estimate. In fact, It follows from \eqref{commIh} and \eqref{error1} that
\begin{equation}\label{Ihuuh}
    b_h({\tau}_h,I_hu-u_h)=b_h({\tau}_h,u-u_h)=-a({\sigma}-{\sigma}_h,{\tau}_h).
\end{equation}
Then using \eqref{infsup0}, \eqref{Ihuuh} and the boundedness of $a$, we obtain
\begin{equation}\label{superu1}
    \|I_hu-u_h\|_{2,h}\lesssim\sup_{{\tau}_h\in {\Sigma}_h, \|\tau_h\|=1}a({\sigma}-{\sigma}_h,{\tau}_h)\lesssim\|{\sigma}-{\sigma}_h\|.
\end{equation}
A combination of Lemma \ref{superu1} and \eqref{apriori2} yields 
\begin{equation}\label{supercloseu}
    \|I_hu-u_h\|_{2,h}\lesssim h^r|\sigma|_r,
\end{equation}
a supercloseness result with respect to the discrete $H^2$-norm which  cannot be derived from   $|u-u_h|_1=O(h^r)$ in \eqref{apriori1}. Let $\delta_{ij}$ denote the Kronecker delta. 
For the norm $|\cdot|_1$, Theorem 4.2 of \cite{Comodi1989} gives a similar supercloseness result  \begin{equation}\label{supercloseu2}
    |I_hu-u_h|_1\lesssim h^{r+1}\big(\|u\|_{r+2}+\delta_{r1}\|f\|\big),
\end{equation}
under the assumption that $\Omega$ is convex. The next theorem confirms the superconvergence property of $u_h^*$ constructed in Section \ref{sec02h} with respect to the norm $|\cdot|_1.$

\begin{theorem}\label{superu}
Let $\Omega$ be a convex domain.
It holds that
\begin{equation*}
    |u-u_h^*|_1\lesssim h^{r+1}\big(\|u\|_{r+2}+\delta_{r1}\|f\|\big).
\end{equation*}
\end{theorem}
\begin{proof}
Let $\tilde{u}_h=I_h^{r+1}u\in U_h^{r+1}$ and $v=(I-I_h)(\tilde{u}_h-u_h^*)\in W_h.$ It follows from \eqref{L2} and the Cauchy--Schwarz inequality that
\begin{equation}\label{Hessv}
\begin{aligned}
    &\|\nabla_h^2v\|^2\simeq(\mathbb{M}\nabla_h^2v,\nabla_h^2v)\\
    &=(\mathbb{M}\nabla_h^2(\tilde{u}_h-u),\nabla_h^2v)+(\mathbb{M}\nabla^2u-\sigma_h,\nabla_h^2v)\\
    &\quad-(\mathbb{M}\nabla_h^2I_h(\tilde{u}_h-u_h^*),\nabla_h^2v)\\
    &\lesssim\big(\|\nabla_h^2(\tilde{u}_h-u)\|+\|\sigma-\sigma_h\|+\|\nabla_h^2I_h(\tilde{u}_h-u_h^*)\|\big)\|\nabla_h^2v\|.
    \end{aligned}
\end{equation}
A combination of $I_h\tilde{u}_h=I_hu$, \eqref{Hessv}, and a scaling argument then yields
\begin{equation}\label{v1}
    |v|_1\lesssim h\|\nabla_h^2v\|\lesssim h\big(\|\nabla^2_h(u-\tilde{u}_h)\|+\|\sigma-\sigma_h\|+ \|\nabla^2_hI_h(u_h-u_h^*)\|\big).
\end{equation}
We conclude the proof with the triangle inequality
\begin{equation*}
|u-u^*_h|_1\leq|u-\tilde{u}_h|_1+|v|_1+|I_h(\tilde{u}_h-u^*_h)|_1,
\end{equation*}
\eqref{v1}, \eqref{apriori2}, \eqref{supercloseu}, \eqref{supercloseu2}, $I_h\tilde{u}_h=I_hu$, and
the classical interpolation error estimate. 
\qed\end{proof}

\subsection{Superconvergence of the moment}
We then derive superconvergence results for the moment variable. Unlike the variable $u$, superconvergence for $\sigma$ depends on the mesh structure and the polynomial degree $r$. In the following, we focus on the lowest order HHJ method ($r=1$).

We reconstruct a new moment field $\sigma_h^*=R_h\sigma_h$ via the postprocessing operator $R_{h}$ based on local least-squares fitting. The operator $R_{h}$ is a linear mapping from ${\Sigma}^1_h$ to the nodal space $[S_h^1]_s^4$. In $\mathcal{T}_h,$ let $\mathcal{V}_h$, $\mathcal{V}^o_h$, $\mathcal{V}^\partial_h$ denote the sets of vertices, interior vertices,  boundary vertices,  respectively. For a closed subdomain $\Omega_0$, let $\mathcal{T}_h(\Omega_0)$, $\mathcal{E}_h(\Omega_0)$ denote the set of elements and edges in $\Omega_0$, respectively. The postprocessed solution $R_h{\sigma}_h$ is determined by nodal values of $R_h{\sigma}_h$ at each vertex in $\mathcal{V}_h$. We say two vertices $z_1, z_2\in\mathcal{V}_h$ are directly connected if they are two endpoints of an edge $e\in\mathcal{E}_h.$ Following \cite{NagaZhang2004},  each vertex $z\in\mathcal{V}_h$ is assigned with a closed and connected Lipschitz vertex patch $\Omega_z\ni z$ as follows. 
\begin{enumerate}
\item[]{Case 1.} For $z\in\mathcal{V}_h^o$, $\Omega_z$ is the union of triangles sharing $z$ as a vertex and possibly a few extra triangles surrounding $z$ but not containing $z$ in $\mathcal{T}_h$. 
\item[]Case 2. For $z\in\mathcal{V}_h^\partial$ directly connected to $z^{\prime}\in\mathcal{V}_h^{o}$, let $\Omega_{z}:=\Omega_{z^{\prime}}$. 
\item[]Case 3. For $z\in\mathcal{V}_h^\partial$ that is not directly connected to any interior vertices, assume that $z$ and $z^\prime\in\mathcal{V}_h^o$ are indirectly connected through a path of edges. Let $\Omega_z$ be the smallest patch containing $z\cup\Omega_{z^\prime}$.
\end{enumerate}

\begin{figure}[ht]
\begin{subfigure}{.3\textwidth}
  \centering
  \includegraphics[width=1\linewidth]{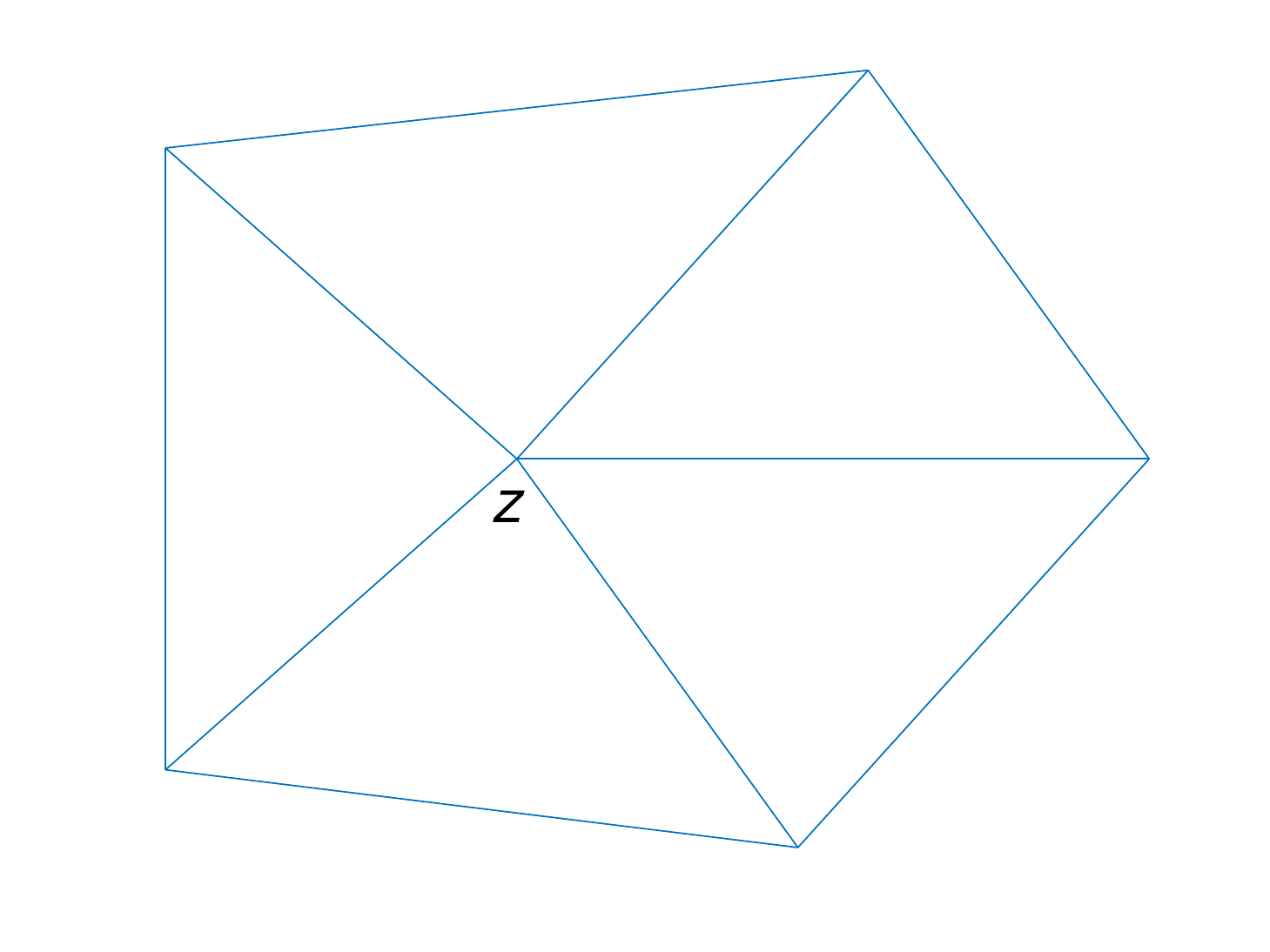}  
  \caption{$\Omega_z$ in Case 1}
  \label{vertexpatch1}
\end{subfigure}
\begin{subfigure}{.3\textwidth}
  \centering
  \includegraphics[width=1\linewidth]{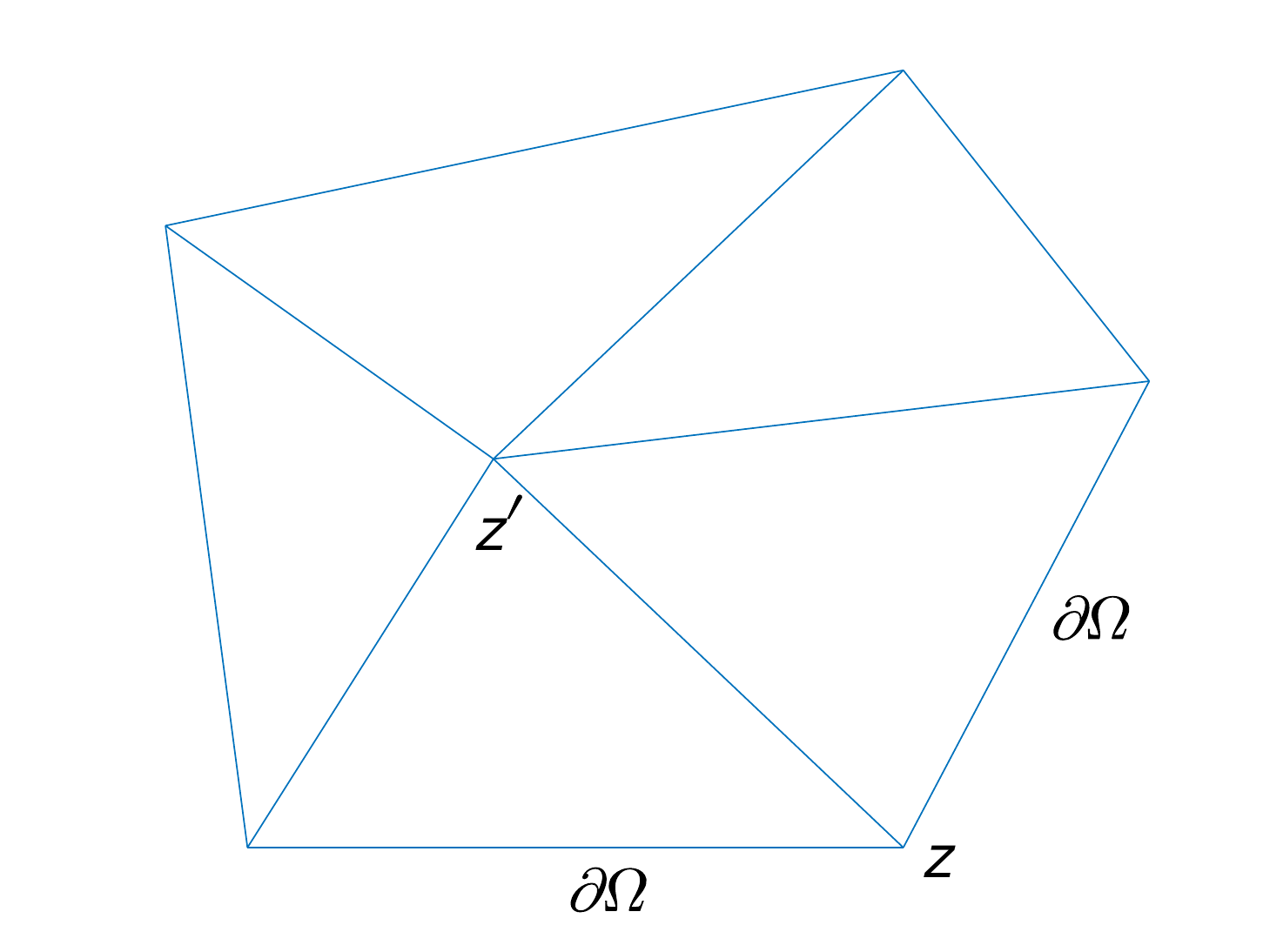}  
  \caption{$\Omega_z$ in Case 2}
  \label{vertexpatch2}
\end{subfigure}
\begin{subfigure}{.3\textwidth}
  \centering
  \includegraphics[width=1\linewidth]{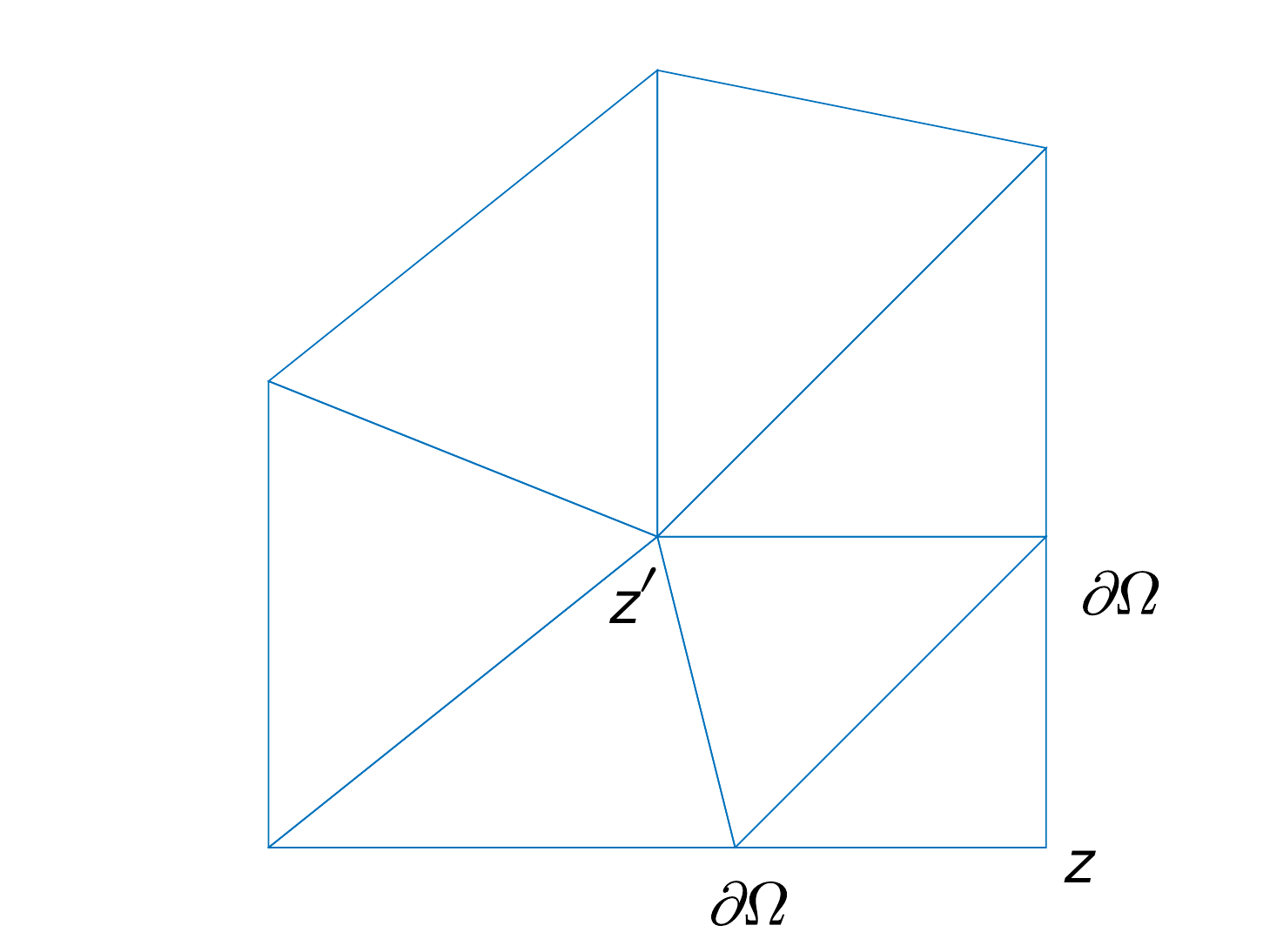}  
  \caption{$\Omega_z$ in Case 3}
  \label{vertexpatch3}
\end{subfigure}
\caption{Vertex patch $\Omega_z$.}
\label{vertexpatch}
\end{figure}
In practical meshes, most vertices belong to Cases 1 and 2, while a few corner points are indirectly connected to interior vertices through a path of two edges. {Examples of the vertex patch $\Omega_z$ at an interior or boundary vertex $z$ are shown in Fig.~\ref{vertexpatch}.} The postprocessing operator $R_h$ is defined as follows.
\begin{definition}\label{Rh} Given ${\tau}_h\in{\Sigma}^1_h,$
the image $R_{h}{\tau}_h\in [S^1_h]_{s}^{4}$ is defined as follows. For each $z\in\mathcal{V}_h$,
let $\tau_z\in[\mathcal{P}_{1}(\Omega_z)]_{s}^{4}$ be the minimizer such that
\begin{equation*}
\tau_z=\arg\min_{\tau\in[\mathcal{P}_{1}(\Omega_z)]_{s}^4}\sum_{e\in\mathcal{E}_h(\Omega_z)}\big([\tau(m_e)]_{nn}-[\tau_h(m_e)]_{nn}\big)^2
\end{equation*} 
where $m_e$ is the midpoint of $e$. Then $R_h\tau_h$ at $z\in\mathcal{V}_h$ is defined as $R_h\tau_h(z):=\tau_{z}(z)$.
\end{definition}

To clarify this postprocessing procedure, we rewrite $R_h$ in linear algebra language. For a vertex $z$, let $\{e_j\}_{j=1}^N$ be the set of edges in the local patch $\Omega_z$. The minimizer $\tau_z=\tau_z(x_1,x_2)$ is of the form
\begin{equation*}
\tau_z=\begin{pmatrix}c_{1}+c_{2}x_{1}+c_{3}x_{2},&c_{4}+c_{5}x_{1}+c_{6}x_{2}\\
c_{4}+c_{5}x_{1}+c_{6}x_{2},&c_{7}+c_{8}x_{1}+c_{9}x_{2}\end{pmatrix},
\end{equation*} 
with parameters $\{c_i\}_{i=1}^9$ to be determined.
Let $m_{j}=(m_{j1},m_{j2})^{\intercal}$ be the midpoint of $e_j$ and $n_{j}=(n_{j1},n_{j2})^{\intercal}$ the unit normal to $e_{j}$. We define 
\begin{equation*}
d_z=(\tau_h(m_1)_{nn}|_{e_1},\ldots,\tau_h(m_N)_{nn}|_{e_N})^{\intercal}
\end{equation*}
and $A_z=(a_1^{\intercal},\ldots,a_{N}^\intercal)^{\intercal}\in\mathbb{R}^{N\times9}$ with 
\begin{align*}
a_{j}=(&n_{j1}^{2}, n_{j1}^{2}m_{j1}, n_{j1}^{2}m_{j2}, 2n_{j1}n_{j2},2n_{j1}n_{j2}m_{j1},\\
&2n_{j1}n_{j2}m_{j2},n_{j2}^{2}, n_{j2}^{2}m_{j1}, n_{j2}^{2}m_{j2}).
\end{align*}
Then
$c_z=(c_{1},\ldots,c_{9})^{\intercal}$ solves $\min_{{\hat{c}}\in\mathbb{R}^{9}}|A_{z}\hat{c}-d_{z}|^2$, where $|\cdot|$ is the Euclidean $l^2$ norm.

In theory and practice, it is important to analyze the unique solvability of those local least-squares problem. If $A_z$ is of full column rank, then $c_z$ is the unique solution of the normal equation $A_z^\intercal A_zc_z=A_z^\intercal d_z.$ 
Given a scalar-valued $v$ and a vector-valued $\phi=(\phi_1,\phi_2)^\intercal,$ define
\begin{align*}
\curl v&:=(-\partial_{x_{2}}v,\partial_{x_{1}}v)^{\intercal},\\
\Curl{\phi}&:=(\curl\phi_{1},\curl\phi_{2})^{\intercal},\\
\Curl^s{\phi}&:=(\Curl{\phi}+(\Curl{\phi})^\intercal)/2.
\end{align*}
The next technical lemma is an important tool in the analysis of $R_h.$ 
\begin{lemma}\label{errexp}
For each $T\in\mathcal{T}_h$ with edges $\{e_k\}_{k=1}^3$, let $\ell_k$ be the length of $e_k$, ${n}_k$ the outward unit normal to $e_k$, ${t}_k$ the counterclockwise unit tangent to $e_k$, $\lambda_k$ the barycentric coordinate at the vertex opposite to $e_k$, $\phi_k=\lambda_{k-1}\lambda_{k+1}$ with $\lambda_0=\lambda_3, \lambda_4=\lambda_1$. Given ${\tau}\in[\mathcal{P}_{1}(T)]_{s}^{4}$ with $\int_{e_k}{\tau}_{nn}ds=0$, $k=1,2,3$, we have
\begin{equation*}
\tau=\Curl^s{r}_{{\tau}},\quad r_\tau=\sum_{k=1}^3\gamma_{k,{\tau}}\phi_k,
\end{equation*}
where  $\{{\gamma}_{k,{\tau}}\}_{k=1}^3$ are given by
\begin{subequations}\label{abk}
\begin{align}
{\gamma}_{k,{\tau}}\cdot{n}_k&=\frac{\ell_k^2}{2}{n}_k^{\intercal}
\partial_{{t}_k}{\tau}{n}_k,\label{normal}\\
{\gamma}_{k,{\tau}}\cdot{t}_k&=\frac{\ell_k^2}{2}{n}_k^{\intercal}
\partial_{{n}_k}{\tau}{n}_k+\ell_k^2{t}_k^\intercal
\partial_{{t}_k}{\tau}{n}_k.\label{tangential}
\end{align}
\end{subequations}
\end{lemma}
\begin{proof}
Let $e_1=(1,0)^\intercal$, $e_2=(0,1)^\intercal$. For the time being, we assume 
\begin{equation}\label{basis}
\big\{\Curl^s(\phi_ie_1)\big\}_{i=1}^3\cup\big\{\Curl^s(\phi_ie_2)\big\}_{i=1}^3
\end{equation}
are linearly independent. The reason will be given at the end of the proof.

Direct calculation shows that for $1\leq i, k\leq3,$
\begin{equation}
\int_{e_k}\big(\Curl^s(\phi_ie_1)\big)_{nn}ds=\int_{e_k}\big(\Curl^s(\phi_ie_2)\big)_{nn}ds=0.
\end{equation}
Then by 
$\int_{e_k}\tau_{nn}ds=0$ with $1\leq k\leq3$ and counting the dimension, we have 
\begin{equation}\label{expsigl}
\tau=\sum_{i=1}^3
\alpha_{i}\Curl^s\big(\phi_ie_1\big)+\sum_{i=1}^{3}\beta_i\Curl^s\big(\phi_ie_2\big)=\Curl^sr_{\tau},
\end{equation}
where $r_{\tau}=\sum_{i=1}^{3}\phi_{i}{\gamma}_{i,{\tau}}$ with $\gamma_{i,{\tau}}:=(\alpha_{i},\beta_{i})^{\intercal}$ an undetermined constant vector.
Given two unit vectors $d_{1}$ and $d_{2}$, it follows from \eqref{expsigl} that
\begin{equation}\label{basicidtt}
d_{1}^\intercal{\tau}d_2
=-\frac{1}{2}\sum_{i=1}^{3}\left(d_1^{\intercal}
{\gamma}_{i,{\tau}}\frac{\partial\phi_i}{\partial d_2^\perp}
+\frac{\partial\phi_i}{\partial d_1^\perp}{\gamma}_{i,{\tau}}^\intercal{d}_2\right),
\end{equation}
where $d_i^\perp$ is the vector from rotating $d_i$ by $\pi/2$ counterclockwise. Recall that $\delta_{ki}$ is the Kronecker delta. 
Applying $\partial_{t_{k}}$ to \eqref{basicidtt} with $d_{1}=d_{2}=n_{k}$,
and using 
\begin{equation}\label{deltaki}
\partial_{t_{k}}^{2}\phi_{i}=-2\delta_{ki}/\ell_k^2,
\end{equation}
we obtain the normal component of $\gamma_{k,\tau}$ in \eqref{normal}. Applying $\partial_{t_k}$ to \eqref{basicidtt} with $d_{1}=t_{k}, d_{2}=n_{k}$ and using \eqref{deltaki}, we have 
\begin{equation}\label{tang1}
t_{k}^{\intercal}\partial_{t_k}{\tau}{n}_{k}=\frac{1}{\ell_k^2}t_k^{\intercal}{\gamma}_{k,\tau}+\frac{1}{2}\sum_{i=1}^{3}\frac{\partial^2\phi_i}{\partial t_k\partial n_k}{\gamma}_{i,\tau}^\intercal n_k.
\end{equation}
Finally $\partial_{n_k}$ to \eqref{basicidtt} with $d_{1}=d_{2}=n_{k}$ leads to  
\begin{equation}\label{tang2}
n_k^{\intercal}\partial_{n_k}\tau n_k=-\sum_{i=1}^{3}\gamma_{i,\tau}^\intercal n_k\frac{\partial^2\phi_i}{\partial t_k\partial n_k}.
\end{equation}
Comparing \eqref{tang2} with \eqref{tang1}, we obtain the tangential component of $\gamma_{k,\tau}$ in \eqref{tangential}. To show the linear independence of \eqref{basis}, suppose
$$\sum_{i=1}^{3}
\alpha_{i}^\prime\Curl^s\big(\phi_{i}e_1\big)+\sum_{i=1}^{3}\beta_{i}^{\prime}\Curl^s\big(\phi_{i}e_2\big)=0.$$ 
Then by running the same argument below \eqref{expsigl}, we obtain that both the normal and tangential components of 
$(\alpha_{i}^{\prime},\beta_{i}^{\prime})$ are zero, i.e., $\alpha_{i}^{\prime}=\beta_{i}^{\prime}=0$. 
\qed\end{proof}

We say two angles are adjacent provided they share a common vertex and are contained in a pair of triangles sharing an edge. 
The next lemma gives a practical criterion for checking the well-posedness of vertex-based least-squares problems in Definition \ref{Rh}. Assumptions in Lemma \ref{uniqueness} were first used in \cite{NagaZhang2004}. It is interesting to see that our least-squares problems are related to the polynomial preserving recovery of linear Lagrange elements introduced in \cite{NagaZhang2004}.
\begin{lemma}\label{uniqueness}
Let $z$ be an interior vertex and $n_z$ the number of grid vertices in $\Omega_z$ that are directly connected to $z$. Assume $n_z\geq4$ and (1): the sum of each pair of adjacent angles in $\Omega_z$ is no greater than $\pi$; and (2): in addition, when $n_z=4$, vertices in $\Omega_z$ are not lying on two lines.
Then there exists a unique $\tau_z\in[\mathcal{P}_1(\Omega_z)]_s^4$ at $z$ in Definition \ref{Rh}. 
\end{lemma}
\begin{proof}
Assume $\int_e(\tau_{z})_{nn}ds=0$ for each $e\in\mathcal{E}_h(\Omega_z)$. Then for each $T\in\mathcal{T}_h(\Omega_z)$, Lemma \ref{errexp} implies 
\begin{equation}\label{taui}
\tau_{z}|_T=\Curl^sr_T,\quad {r}_T=\sum_{k=1}^3\gamma_{k,{\tau_{z}|_T}}\phi_{k}|_T\in[\mathcal{P}_2(T)]^2.
\end{equation} 
Let $r_z$ be the piecewise quadratic polynomial on $\Omega_z$ with $r_z|_T=r_T$, $\forall T\in\mathcal{T}_h(\Omega_z)$. We claim that $r_z$ is indeed a quadratic polynomial on $\Omega_z$. For each interior edge $e\in\mathcal{E}_h(\Omega_z)$ shared by $T, T^{\prime}\in\mathcal{T}_h(\Omega_{z})$, the explicit formulas for $\gamma_{k,{\tau}_{z}|_T}$, $\gamma_{k,{\tau}_{z}|_{T^\prime}}$ in Lemma \ref{errexp} imply that $r_T=r_{T^{\prime}}$ on $e$ and
\begin{equation}\label{tcts}
\partial_{t_e}{r_T}=\partial_{t_e}{r_{T^{\prime}}}\text{ on }e,\quad\partial_{t_e}^2{r_T}=\partial_{t_e}^2{r_{T^{\prime}}}\text{ on } e.
\end{equation}
Therefore $r_z$ is continuous. A direct consequence of \eqref{taui} is
\begin{equation}\label{normalcts}
\begin{aligned}
&t_e^\intercal\partial_{n_e}{r_T}=t_e^\intercal{\tau}_{z}t_e\text{ on }T,\quad
t_e^\intercal\partial_{n_e}{r_{T^{\prime}}}=t_e^{\intercal}{\tau}_zt_e\text{ on } T^{\prime},\\
&\frac{1}{2}\big(n_e^\intercal\partial_{n_e}{r_T}-t_{e}^{\intercal}\partial_{t_e}{r_T}\big)=t_e^{\intercal}\tau_zn_e\text{ on }T,\\
&\frac{1}{2}\big(n_{e}^{\intercal}\partial_{n_e}{r_{T^{\prime}}}-t_e^{\intercal}\partial_{t_e}{r_{T^{\prime}}}\big)=t_{e}^{\intercal}{\tau}_{z}n_e\text{ on }T^{\prime},
\end{aligned}
\end{equation} 
which leads to
\begin{equation}\label{ncts}
    \partial_{n_e}r_T=\partial_{n_e}r_{T^{\prime}}\text{ on }e.
\end{equation}
Applying $\partial_{t_e}$ to \eqref{normalcts} and using \eqref{tcts} yield 
\begin{equation}\label{tncts}
    \partial_{t_e}\partial_{n_e}r_T=\partial_{t_e}\partial_{n_e}r_{T^{\prime}}\text{ on }e.
\end{equation}
Applying $\partial_{n_e}$ to \eqref{normalcts} and using \eqref{tncts} yield 
\begin{equation}\label{nncts}
    \partial_{n_e}^2r_T=\partial_{n_e}^2r_{T^{\prime}}\text{ on }e.
\end{equation}
Therefore collecting \eqref{tcts} and \eqref{ncts}--\eqref{nncts}, we confirm  $r_z\in[\mathcal{P}_2(\Omega_{z})]^2$. 
By the definition, $r_z$ vanishes at all vertices in $\mathcal{T}_h(\Omega_z)$. 
It then follows from the given assumptions on $\Omega_z$ and Theorem 2.3 of \cite{NagaZhang2004} that $r_z=0$ and thus ${\tau}_z=\Curl^sr_z=0$.

In summary, we conclude that $\int_e({\tau}_z)_{nn}ds=0$ $\forall e\in\mathcal{E}_h(\Omega_z)$ implies ${\tau}_z={0}$. In other words, for $c\in\mathbb{R}^9,$ ${A}_z{c}={0}$ implies ${c}=0$, and there exists a unique least-squares solution ${\tau}_z$ at $z$.
\qed\end{proof}

\begin{figure}[ht]
\begin{subfigure}{.48\textwidth}
  \centering
  \includegraphics[width=.8\linewidth]{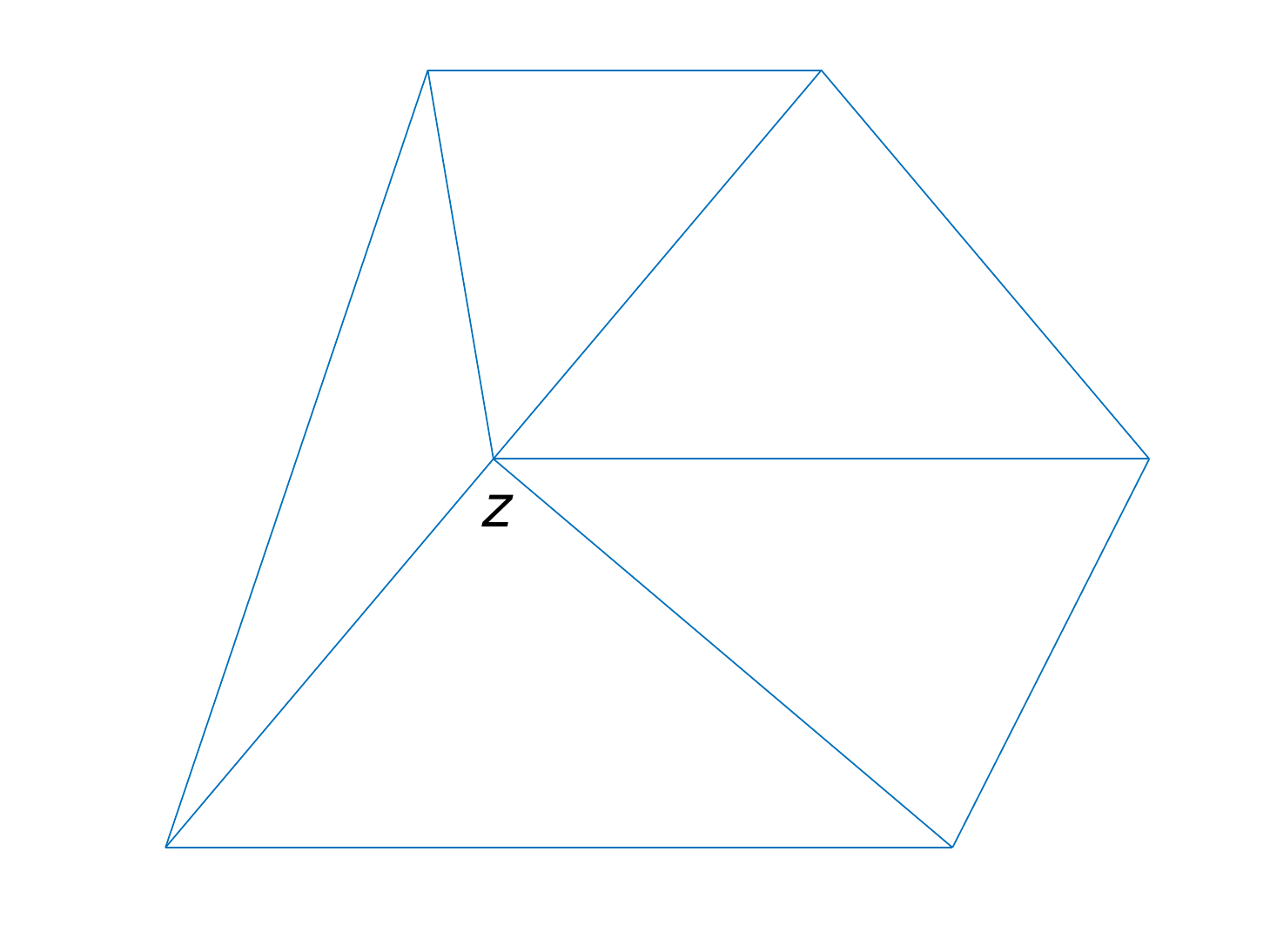}  
  \caption{Assumption (1) is violated.}
  \label{effeta11}
\end{subfigure}
\begin{subfigure}{.48\textwidth}
  \centering
  \includegraphics[width=.8\linewidth]{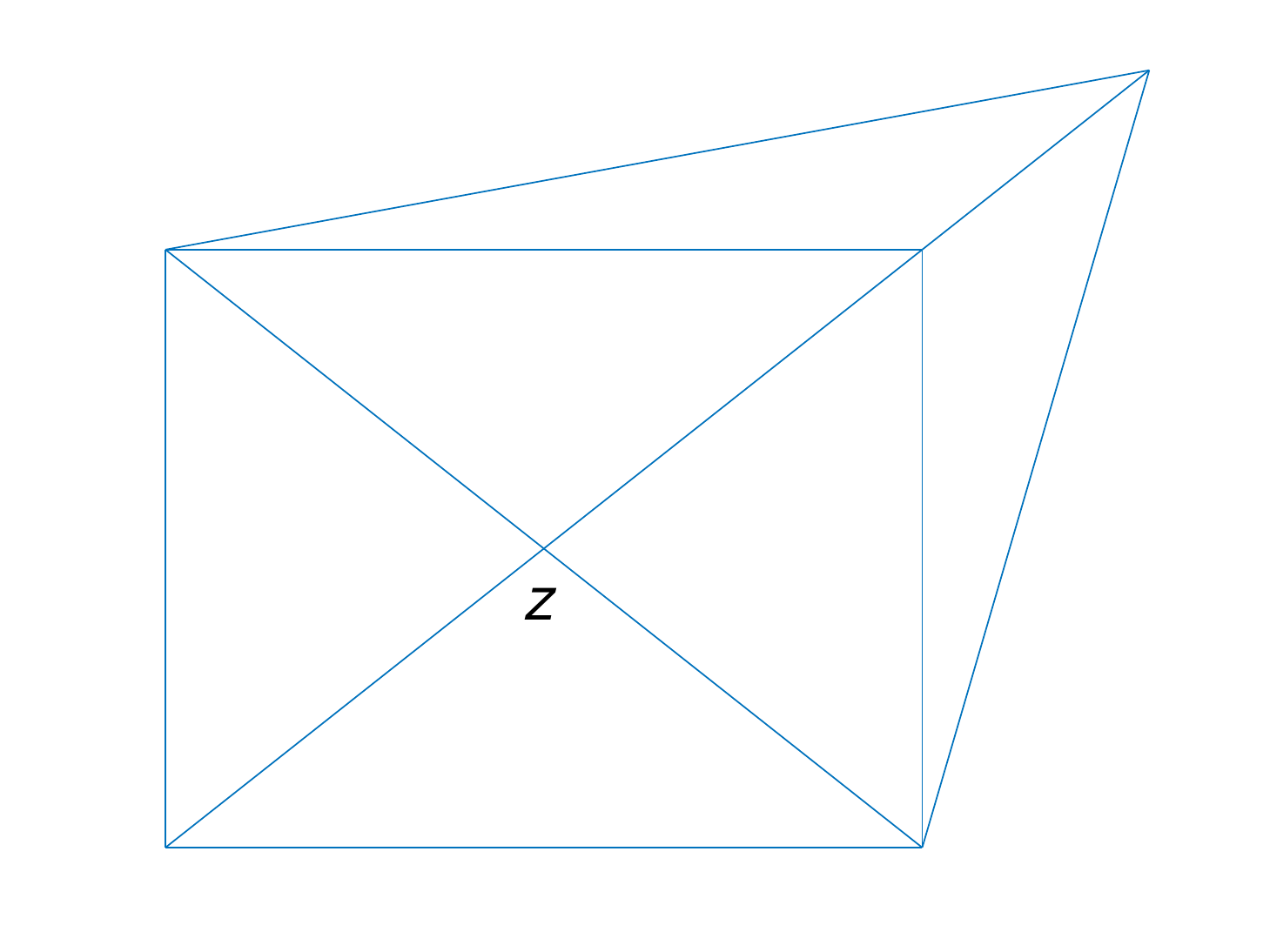}  
  \caption{Assumption (2) is violated.}
  \label{effzeta11}
\end{subfigure}
\caption{{Vertex patches $\Omega_z$ violating assumptions in Lemma} \ref{uniqueness}.}
\label{counterexample}
\end{figure}
For each interior vertex $z$,
we start with a small vertex patch $\Omega_z$, e.g., the union of triangles sharing $z$. In a few cases when the assumption in Lemma \ref{uniqueness} is violated, the local least-squares problem may not have a unique solution, {see Fig.~\ref{counterexample} for examples}. To guarantee the uniqueness, one could add some extra elements to the patch $\Omega_z$, e.g., enlarge $\Omega_z$ by one layer. 

By definitions of $\Omega_z$ and $R_h$, the fitting polynomial ${\tau}_z$ is unique at each boundary vertex $z$ provided the condition in Lemma \ref{uniqueness} holds for all internal vertices. In theory, $R_h$ could be applied to more general functions than members of $\Sigma_h^1.$ For $T\in\mathcal{T}_h,$ let $\Omega_T=\bigcup\{T^{\prime}\in\mathcal{T}_h: \overline{T}^{\prime}\cap \overline{T}\neq\emptyset\}.$
The well-posedness of least-squares problems implies $R_{h}{\tau}=\tau$ on $T$ for ${\tau}\in[\mathcal{P}_1(\Omega_T)]_{s}^{4}$, which is called the ``polynomial preserving property'' in \cite{NagaZhang2004}. The super-approximation property of $R_h$ then follows from a local scaling argument, see the proof of Theorem 3.3 in \cite{NagaZhang2004} or Theorems 2.2 and 2.3 of \cite{BankLi2019} for details.
\begin{theorem}\label{superRh}
Assume that the condition in Lemma \ref{uniqueness} holds for all internal vertices in $\mathcal{V}_h^o$. Then given ${\tau}_h\in{\Sigma}^1_h$ and $T\in\mathcal{T}_h$,  we have
\begin{subequations}
\begin{align}
\|R_h{\tau}_h\|_T&\lesssim
\|{\tau}_h\|_{\Omega_T},\\
\|{\sigma}-R_h{\sigma}\|_T&\lesssim h_T^2|{\sigma}|_{2,\Omega_T}.
\end{align}
\end{subequations}
\end{theorem}

Another ingredient in the superconvergence analysis for the moment variable is the following supercloseness estimate on a uniform grid, see \cite{HuMaMa2021}.
\begin{lemma}\label{superclosesigma}
Assume that each pair of directly adjacent triangles in $\mathcal{T}_h$ forms a parallelogram. When $r=1,$ it holds that
\begin{equation*}
    \|\Pi_h{\sigma}-{\sigma}_h\|\lesssim h^2|\log h|^\frac{1}{2}\big(|{\sigma}|_{\frac{5}{2},\Omega}+\|\nabla{\sigma}\|_{L^{\infty}(\Omega)}\big).
\end{equation*}
\end{lemma}
Although Lemma \ref{superclosesigma} is proved only under exactly uniform grids, such supercloseness estimate could often be extended to mildly structured meshes, see, e.g., \cite{BankXu2003a,XuZhang2004,Li2018SINUM,BankLi2019,Li2021JSC} for similar results in nodal, edge and nonconforming elements in $\mathbb{R}^2$. We also refer to \cite{WuZhang2007} for a supercloseness estimate of nodal elements on graded meshes. Now we are in a position to present the last main result. The proof directly follows from the triangle inequality
\begin{equation*}
\|{\sigma}-R_h\sigma_h\|\leq\|\sigma-R_h\sigma\|+\|R_h(\Pi_h\sigma-\sigma_h)\|
\end{equation*}
and Theorem \ref{superRh} and Lemma \ref{superclosesigma}.
\begin{theorem}\label{superRhsigmah}
Let the assumptions in Theorem \ref{superRh} and Lemma \ref{superclosesigma} hold. Then for $r=1$ we have
\begin{equation*}
    \|{\sigma}-R_{h}{\sigma}_{h}\|\lesssim h^2|\log h|^\frac{1}{2}\big(\|{\sigma}\|_{\frac{5}{2},\Omega}+\|\nabla{\sigma}\|_{L^{\infty}(\Omega)}\big).
\end{equation*}
\end{theorem}

Although the proof of Theorem \ref{superRhsigmah} utilizes the result in Lemma \ref{superclosesigma}, we will show in the experiment that there exists apparent superconvergence of $\|{\sigma}-R_{h}{\sigma}_{h}\|$ while $\|\Pi_h{\sigma}-{\sigma}_{h}\|$ is not super-small at all, which suggests that supercloseness estimate is not a necessary condition for achieving postprocessing superconvergence.
\section{Numerical experiments}\label{secNE}
In this section, we test the performance of the error indicator $\eta_h(T)$ in Section \ref{sec02h}
and the estimator $\zeta_h=\big(\sum_{T\in\mathcal{T}_h}\zeta_h(T)^2\big)^\frac{1}{2}$ with $\zeta_h(T)=\big(\|\sigma_h-R_h\sigma_h\|^2_T+\|\nabla(u_h-u_h^*)\|_T^2\big)^\frac{1}{2}$ in Section \ref{sec01}. The numerical schemes are implemented in MATLAB R2020a. In all experiments we set $Ed^3/12=1$, the Poisson ratio $\nu=0.3$, and use the lowest order HHJ method \eqref{HHJ} ($r=1$). The adaptive algorithm  is based on the classical loop (cf.~\cite{Dorfler1996,NochettoSiebertVeeser2009})
\begin{equation*}
\begin{CD}
\textsf{Solve}@>>>\textsf{Estimate}@>>>\textsf{Mark}@>>>\textsf{Refine}.
\end{CD}
\end{equation*}
The module \textsf{Estimate} calculates  element-wise error indicators $\{\eta_h(T)\}_{T\in\mathcal{T}_h}$ (resp.~$\{\zeta_h(T)\}_{T\in\mathcal{T}_h}$) in the current {grid} $\mathcal{T}_h$.
The module \textsf{Mark} selects a minimal subset of elements $\mathcal{M}_h\subset\mathcal{T}_h$ satisfying
\begin{equation*}
    \sum_{T\in\mathcal{M}_h}\eta_h(T)^2\geq0.6\sum_{T\in\mathcal{T}_h}\eta_h(T)^2\quad\big(\text{resp.}\sum_{T\in\mathcal{M}_h}\zeta_h(T)^2\geq0.6\sum_{T\in\mathcal{T}_h}\zeta_h(T)^2\big).
\end{equation*}
The module \textsf{Refine} subdivides elements in $\mathcal{M}_h$ and minimal neighboring elements by the newest vertex bisection and outputs a new conforming grids, over which \eqref{HHJ} is solved and the element-wise errors are estimated again. The convergence of $\eta_h$-based adaptive algorithm is measured by the error $E_h:=\big(\|\sigma-\sigma_h\|^2+\|u-u_h^*\|_{2,h}^2\big)^\frac{1}{2}$ while $\zeta_h$-based adaptive algorithm is measured by $e_h:=\big(\|\sigma-\sigma_h\|^2+|u-u_h|_1^2\big)^\frac{1}{2}$. 
By $N$ we denote the  number of triangles in the current mesh. We compute the order of convergence $p$ by the MATLAB function \textsf{polyfit} such that the corresponding error is proportional to $N^p$.

\begin{figure}[ht]
\begin{subfigure}{.3\textwidth}
  \centering
  \includegraphics[width=1\linewidth]{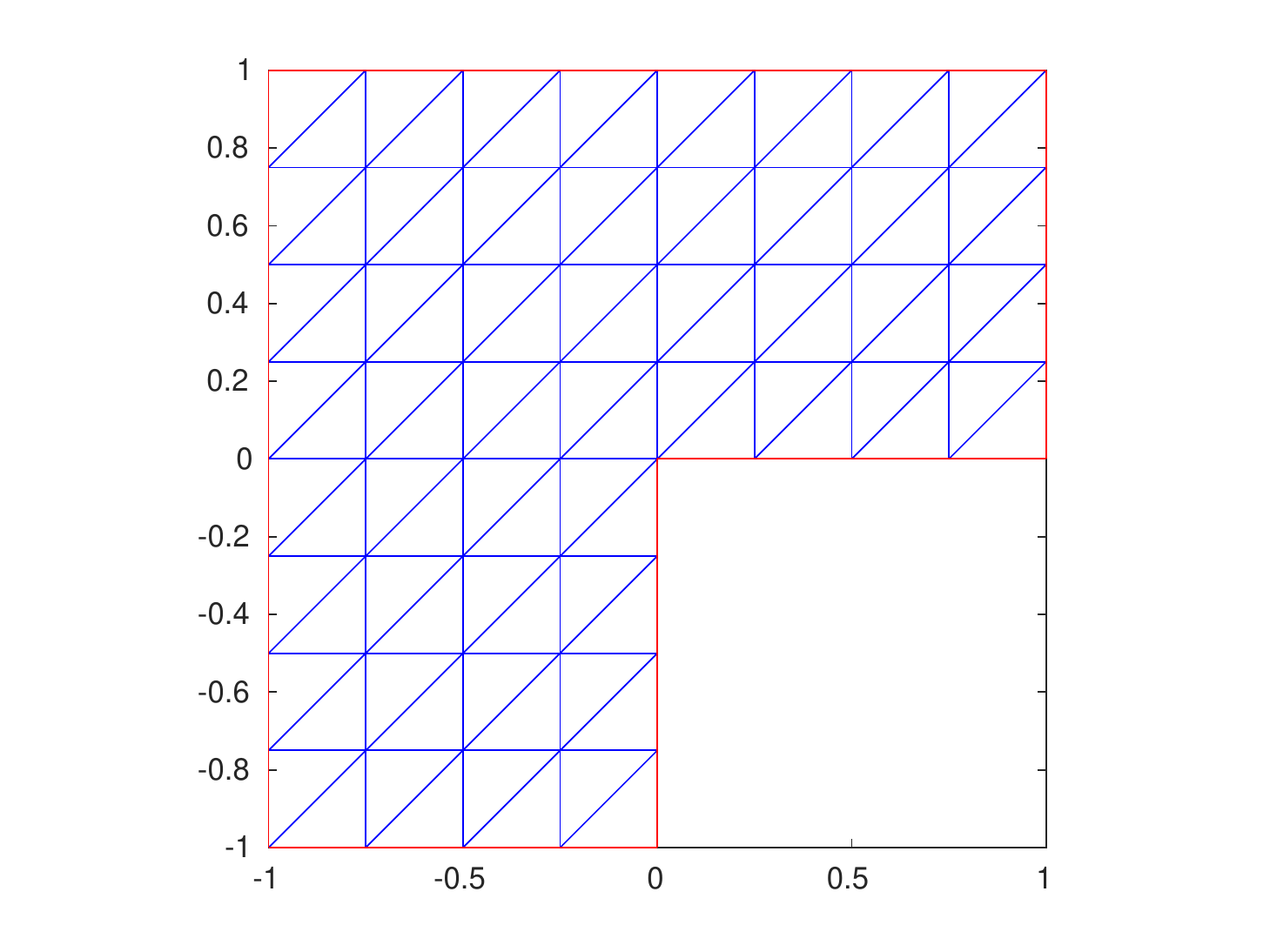}  
  \caption{Initial grid.}
  \label{initialL}
\end{subfigure}
\begin{subfigure}{.3\textwidth}
  \centering
  \includegraphics[width=1\linewidth]{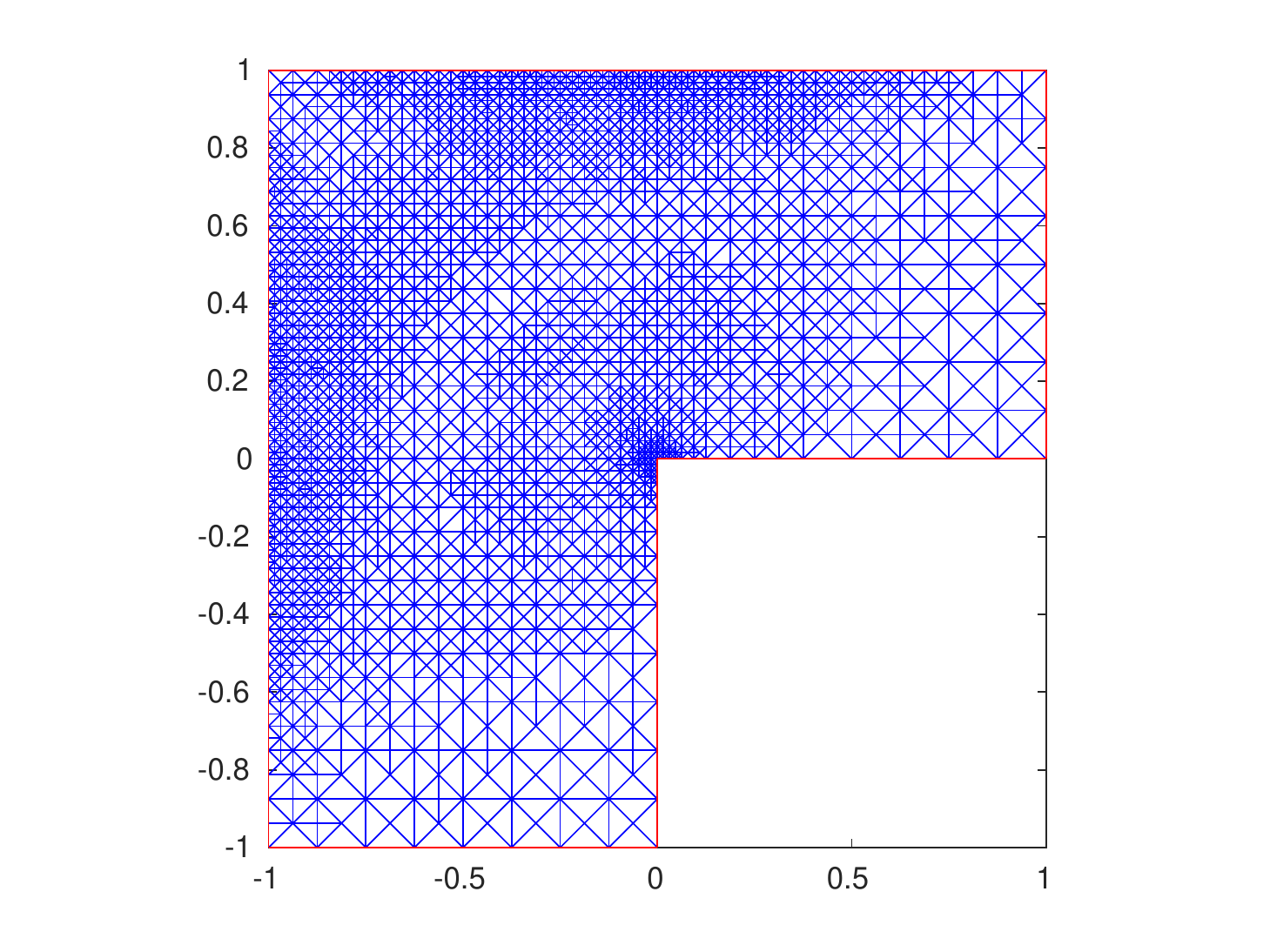}  
  \caption{{A} graded grid by $\eta_h$.}
  \label{graded1}
\end{subfigure}
\begin{subfigure}{.3\textwidth}
  \centering
  \includegraphics[width=1\linewidth]{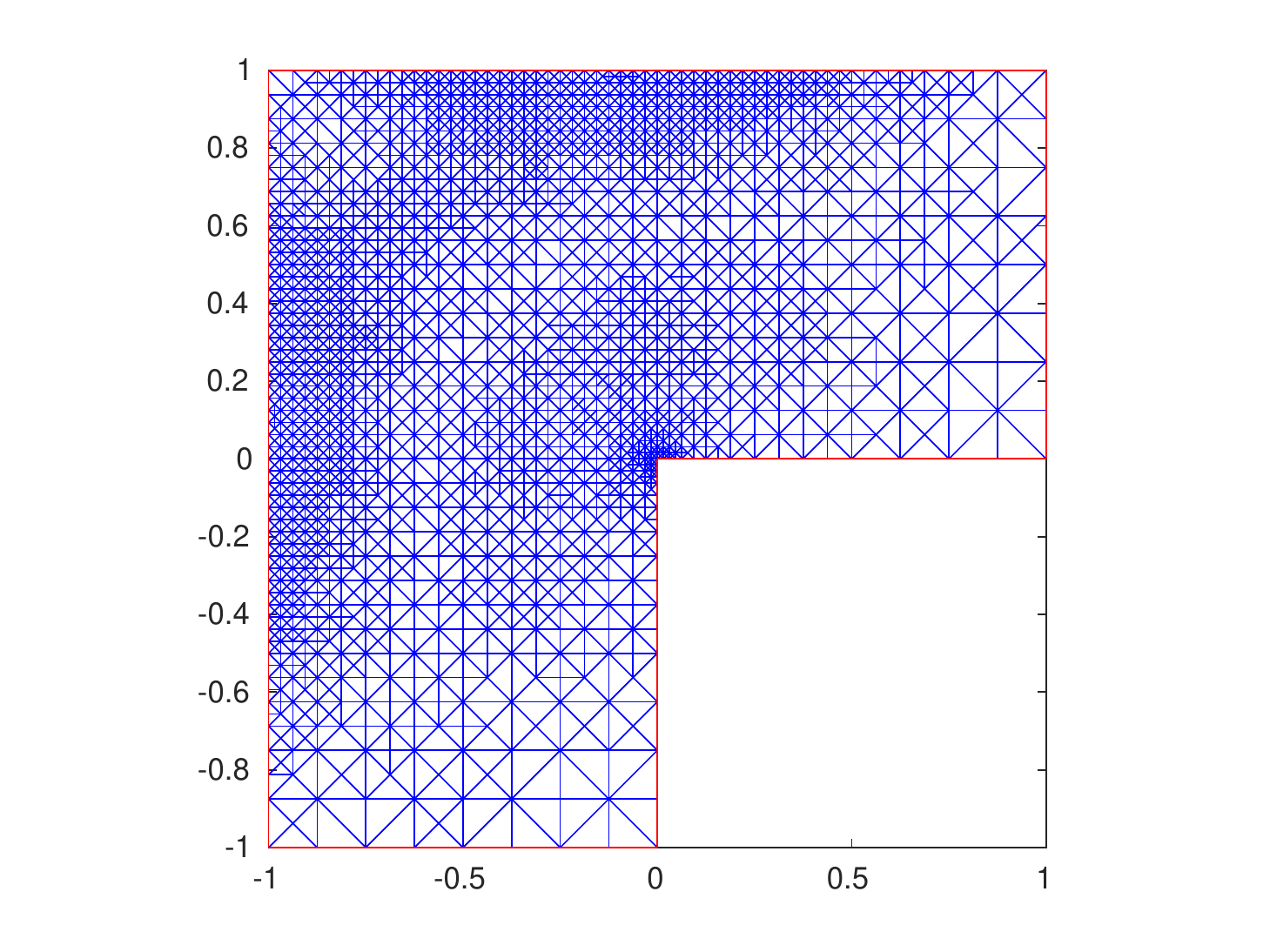}  
  \caption{{A} graded grid by $\zeta_h$.}
  \label{graded2}
\end{subfigure}
\caption{Grids on the L-shaped domain  {for Problem 1}.}
\label{gridclamped}
\end{figure}

\begin{figure}[ht]
\begin{subfigure}{.49\textwidth}
  \centering
  \includegraphics[width=.8\linewidth]{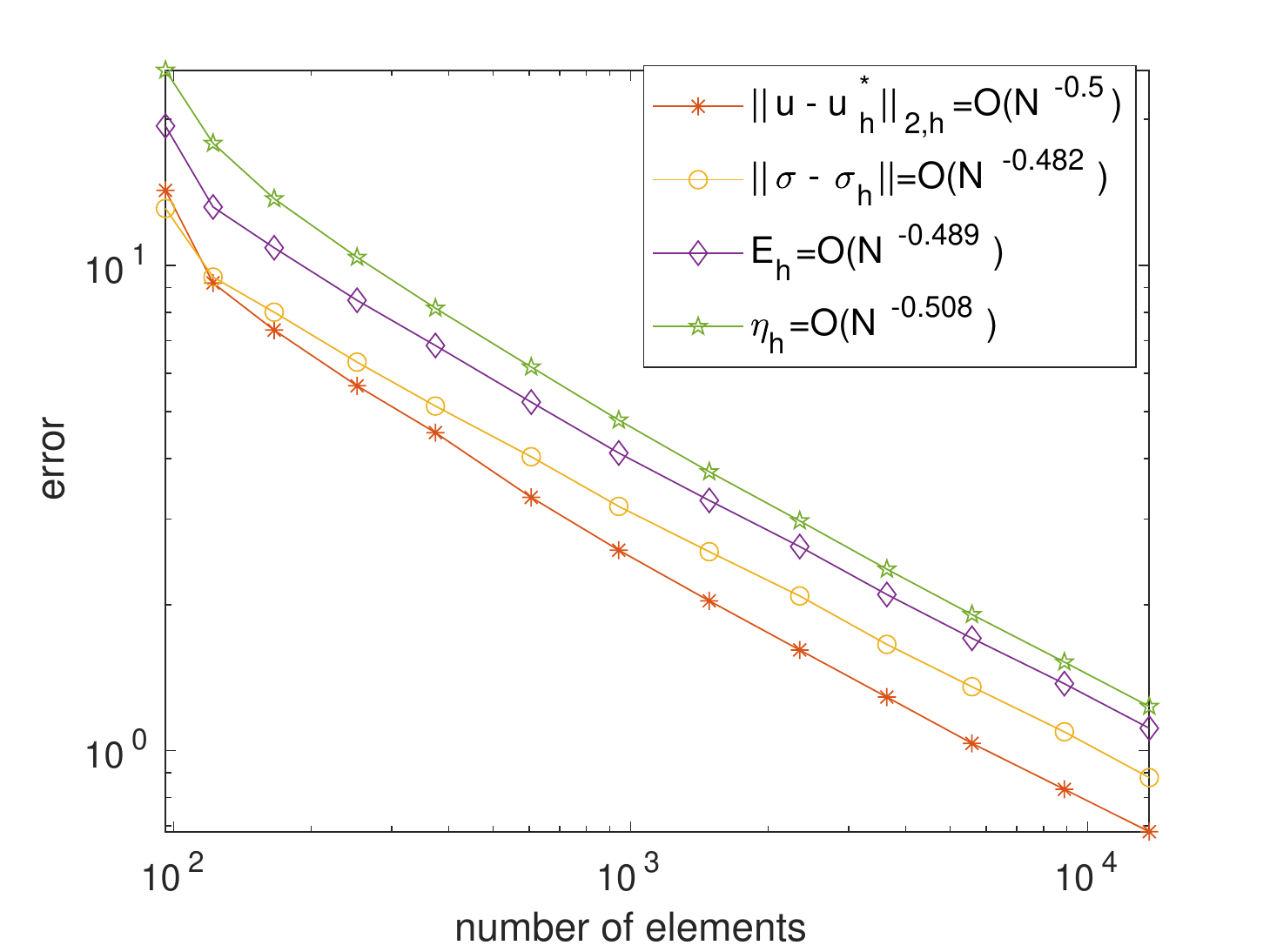}  
  \caption{Convergence history based on $\eta_h$.}
  \label{curveeta}
\end{subfigure}
\begin{subfigure}{.49\textwidth}
  \centering
  \includegraphics[width=.8\linewidth]{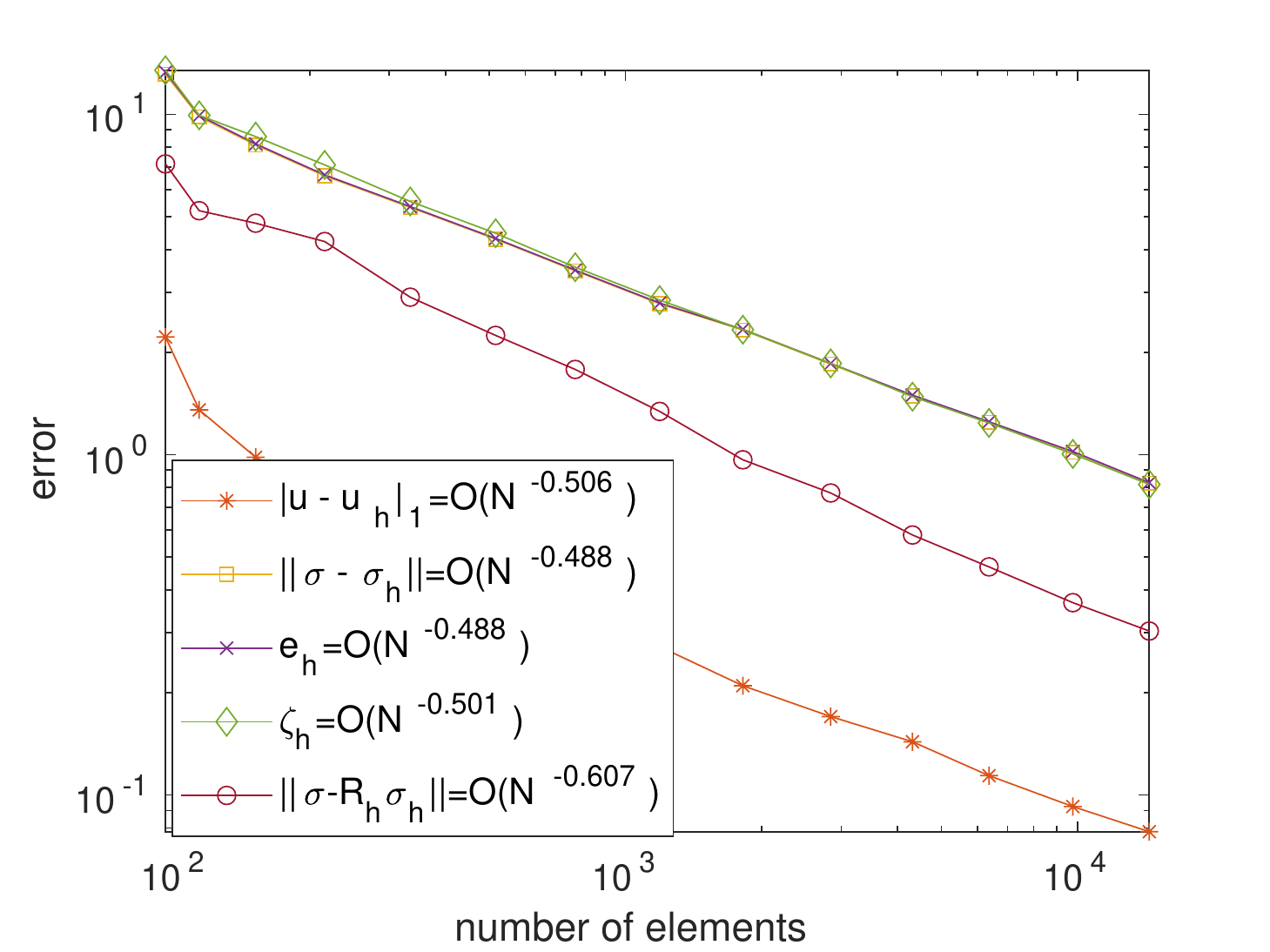}  
  \caption{Convergence history based on $\zeta_h$.}
  \label{curvezeta}
\end{subfigure}
\caption{Convergence history  {of exact errors and estimators for Problem 1}.}
\label{curveclamped}
\end{figure}

\begin{figure}[ht]
\begin{subfigure}{.49\textwidth}
  \centering
  \includegraphics[width=.8\linewidth]{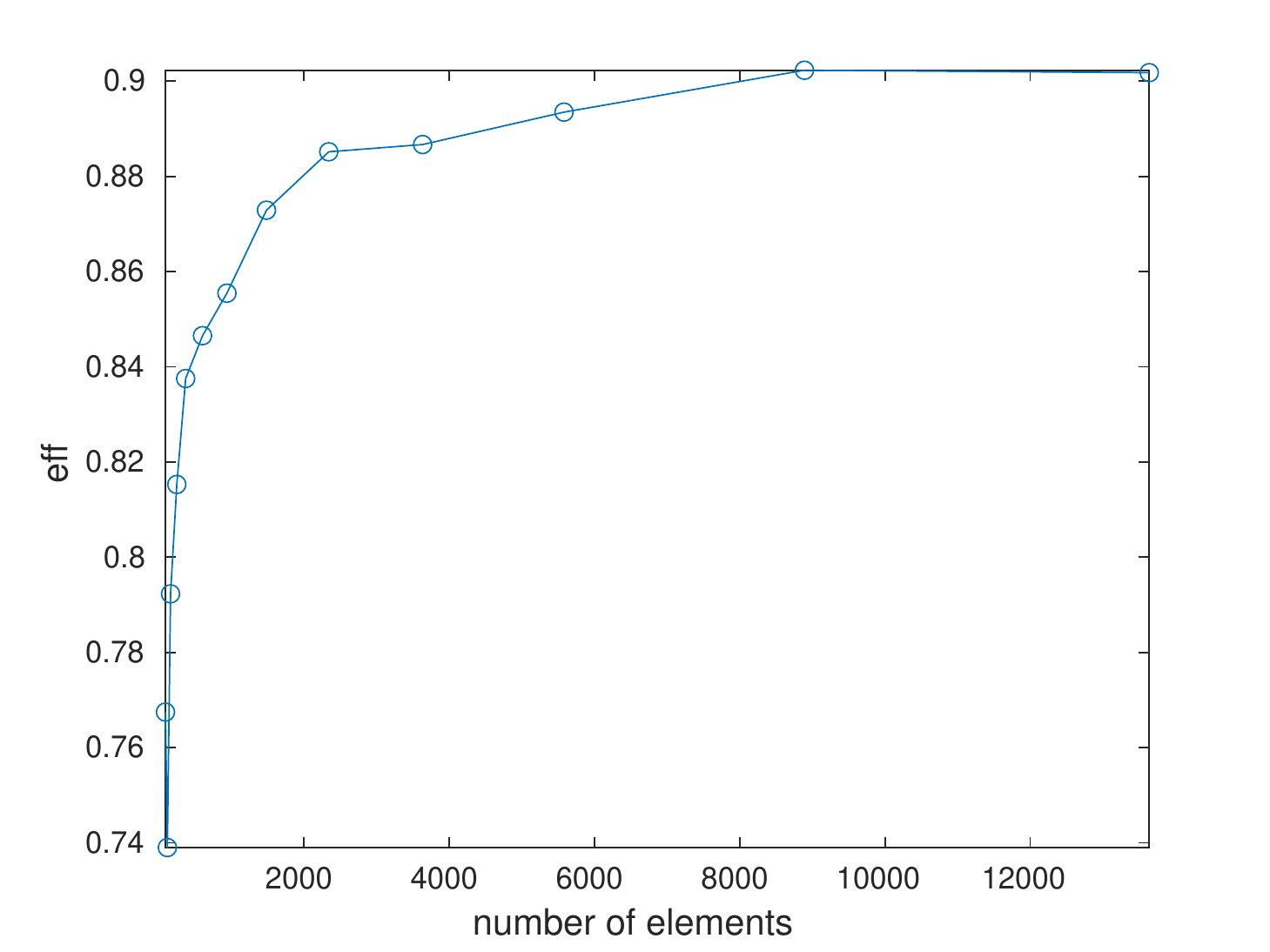}  
  \caption{Effectiveness of $\eta_h$.}
  \label{effeta}
\end{subfigure}
\begin{subfigure}{.49\textwidth}
  \centering
  \includegraphics[width=.8\linewidth]{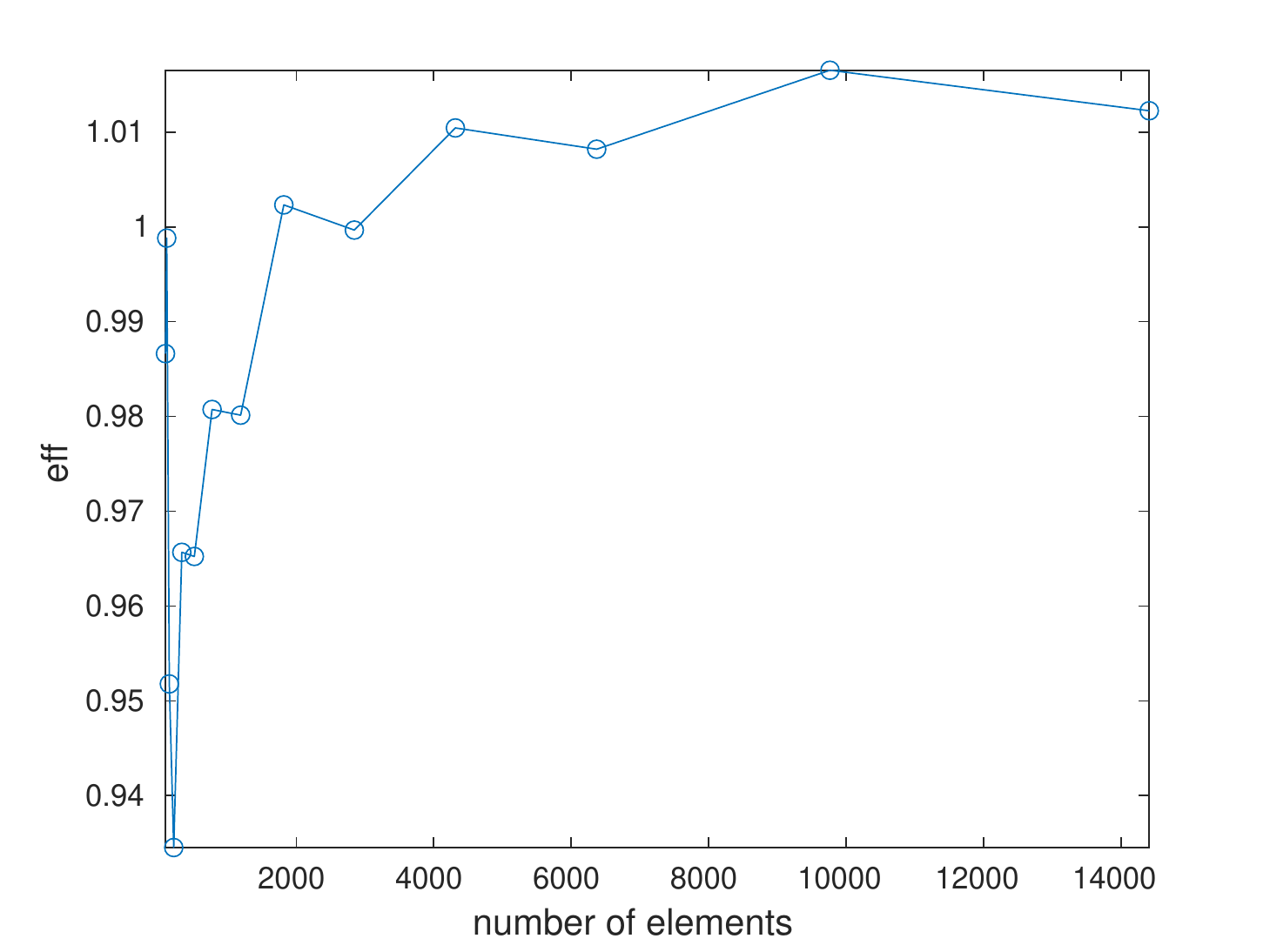}  
  \caption{Effectiveness of $\zeta_h$.}
  \label{effzeta}
\end{subfigure}
\caption{Effectiveness ratio   {of error estimators for Problem 1}.}
\label{eff}
\end{figure}

\begin{figure}[ht]
\begin{subfigure}{.49\textwidth}
  \centering
  \includegraphics[width=.8\linewidth]{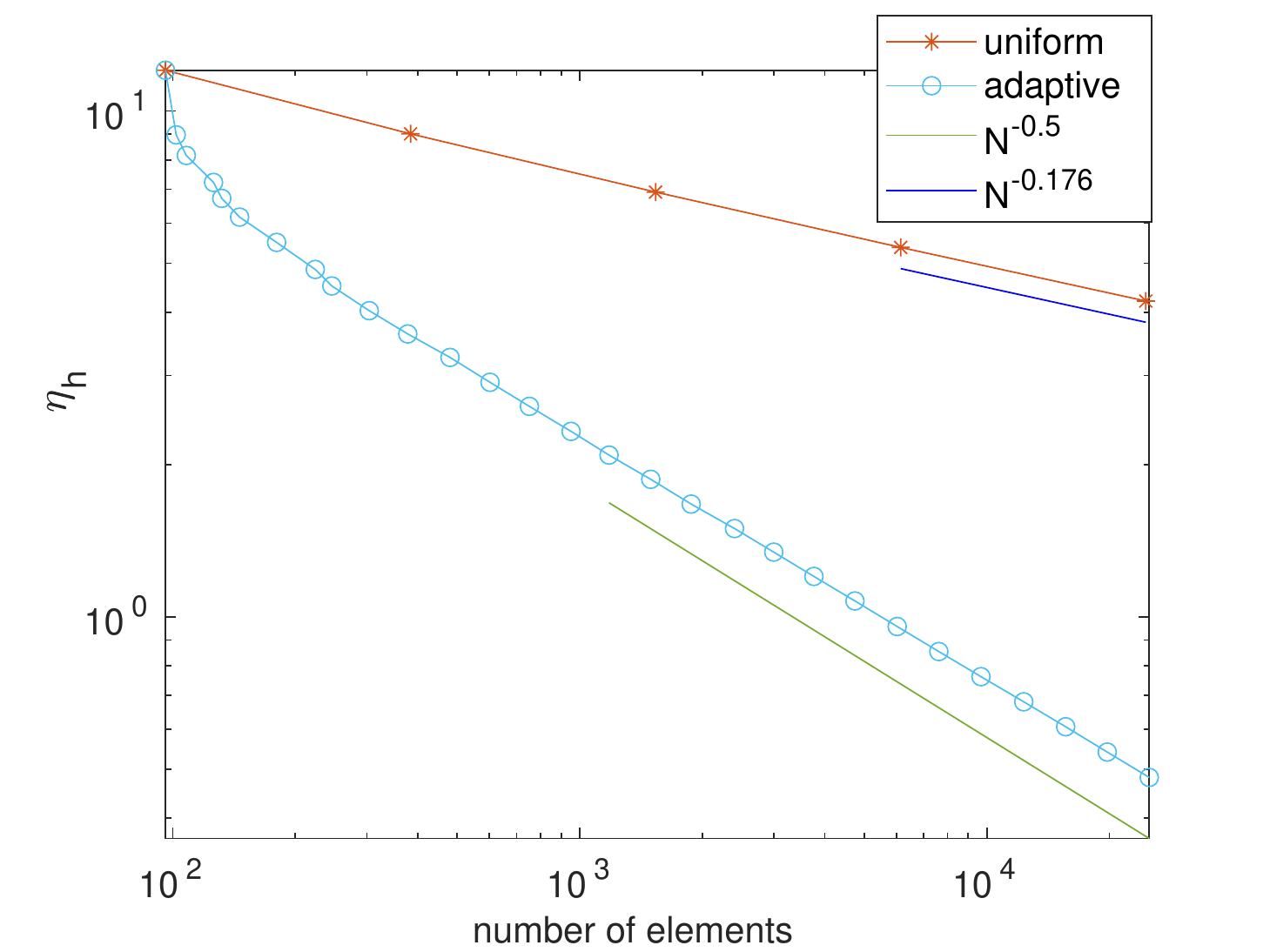}  
  \caption{Convergence of $\eta_h$.}
  \label{eta}
\end{subfigure}
\begin{subfigure}{.49\textwidth}
  \centering
  \includegraphics[width=.8\linewidth]{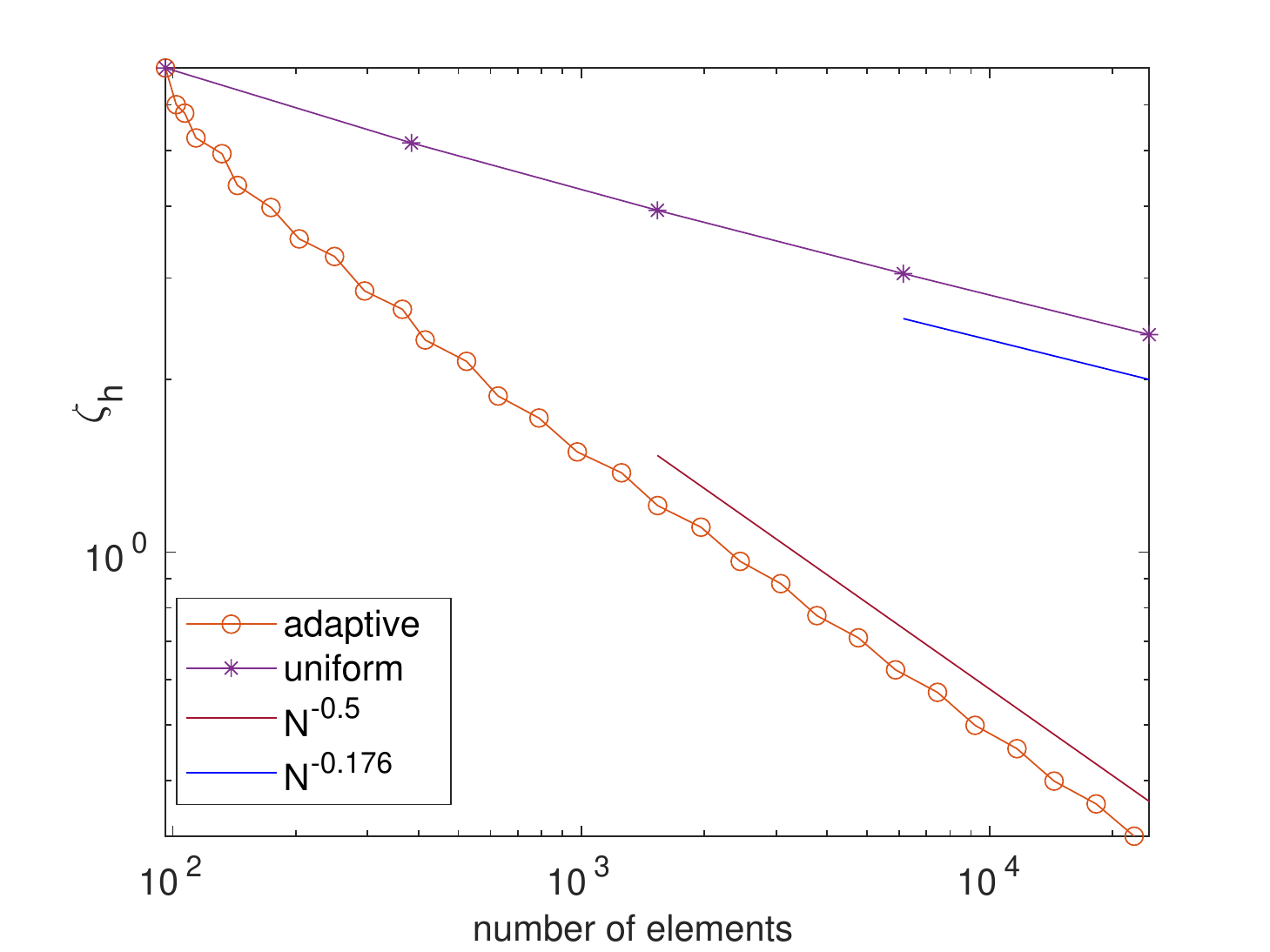}  
  \caption{Convergence of $\zeta_h$.}
  \label{zeta}
\end{subfigure}
\caption{Convergence of error estimators {for Problem 2}.}
\label{estimator}
\end{figure}

\begin{figure}[ht]
\begin{subfigure}{.3\textwidth}
  \centering
  \includegraphics[width=1\linewidth]{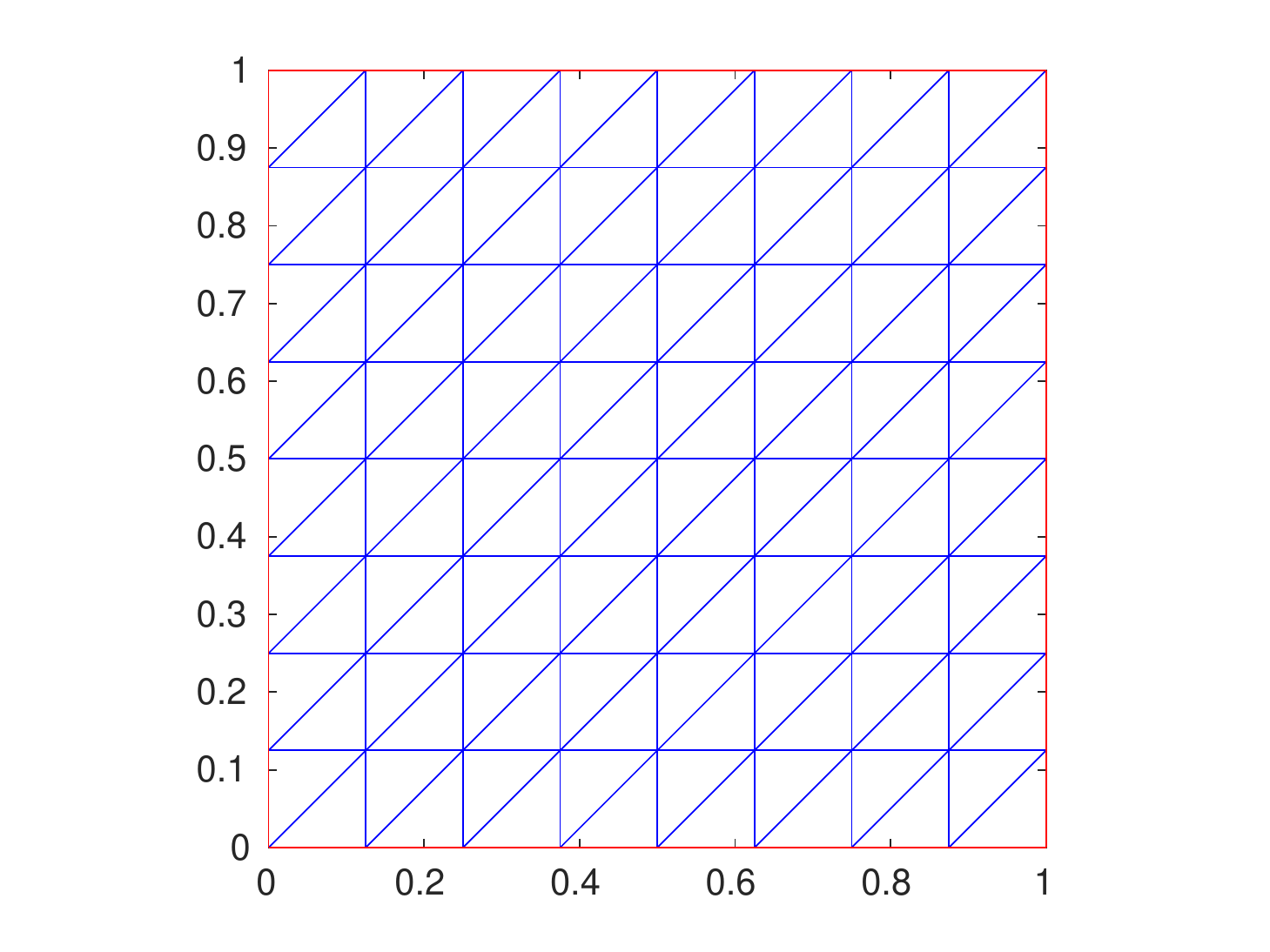}  
  \caption{128 elements}
\end{subfigure}
\begin{subfigure}{.3\textwidth}
  \centering
  \includegraphics[width=1\linewidth]{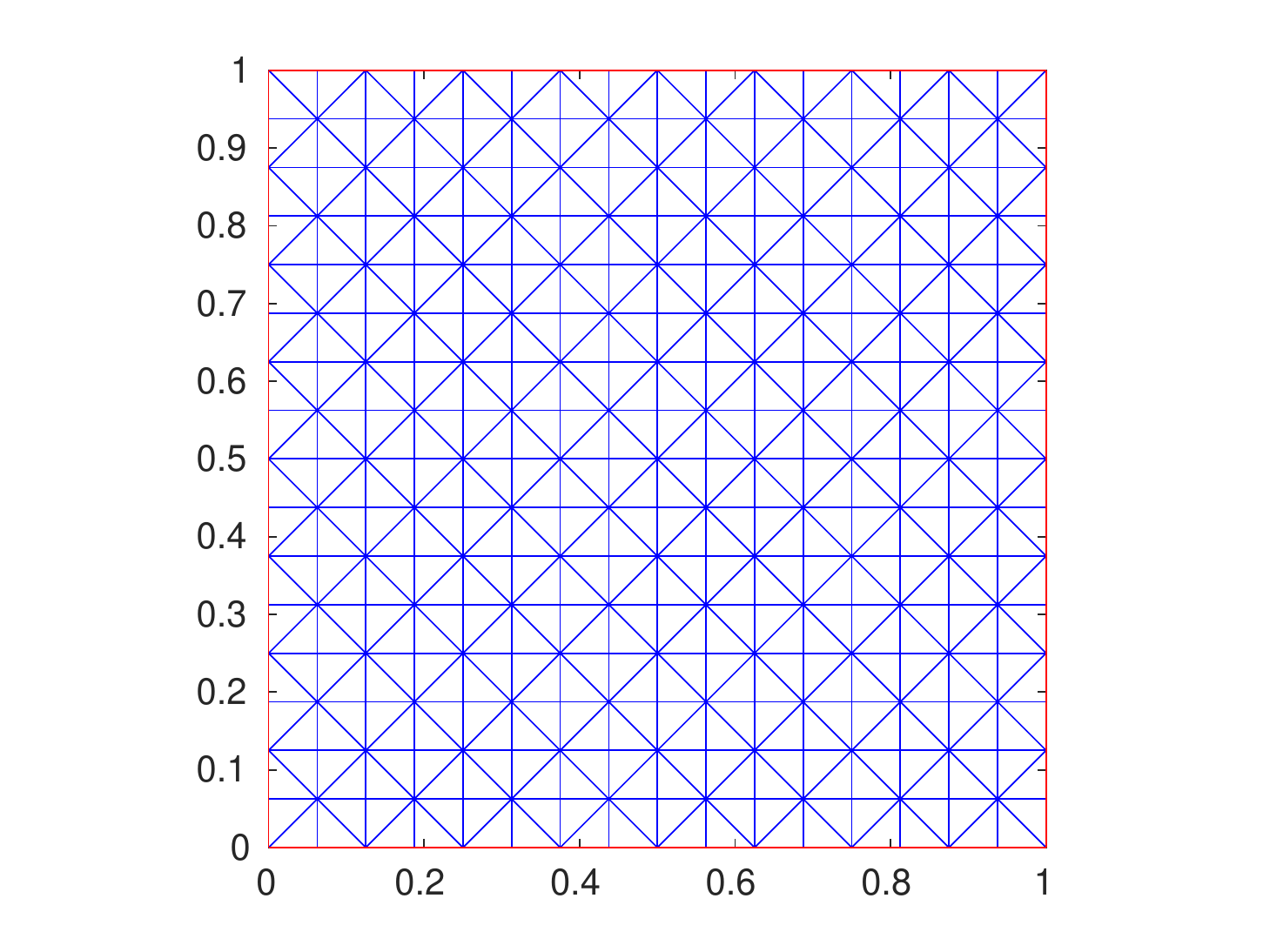}  
  \caption{512 elements}
\end{subfigure}
\begin{subfigure}{.3\textwidth}
  \centering
  \includegraphics[width=1\linewidth]{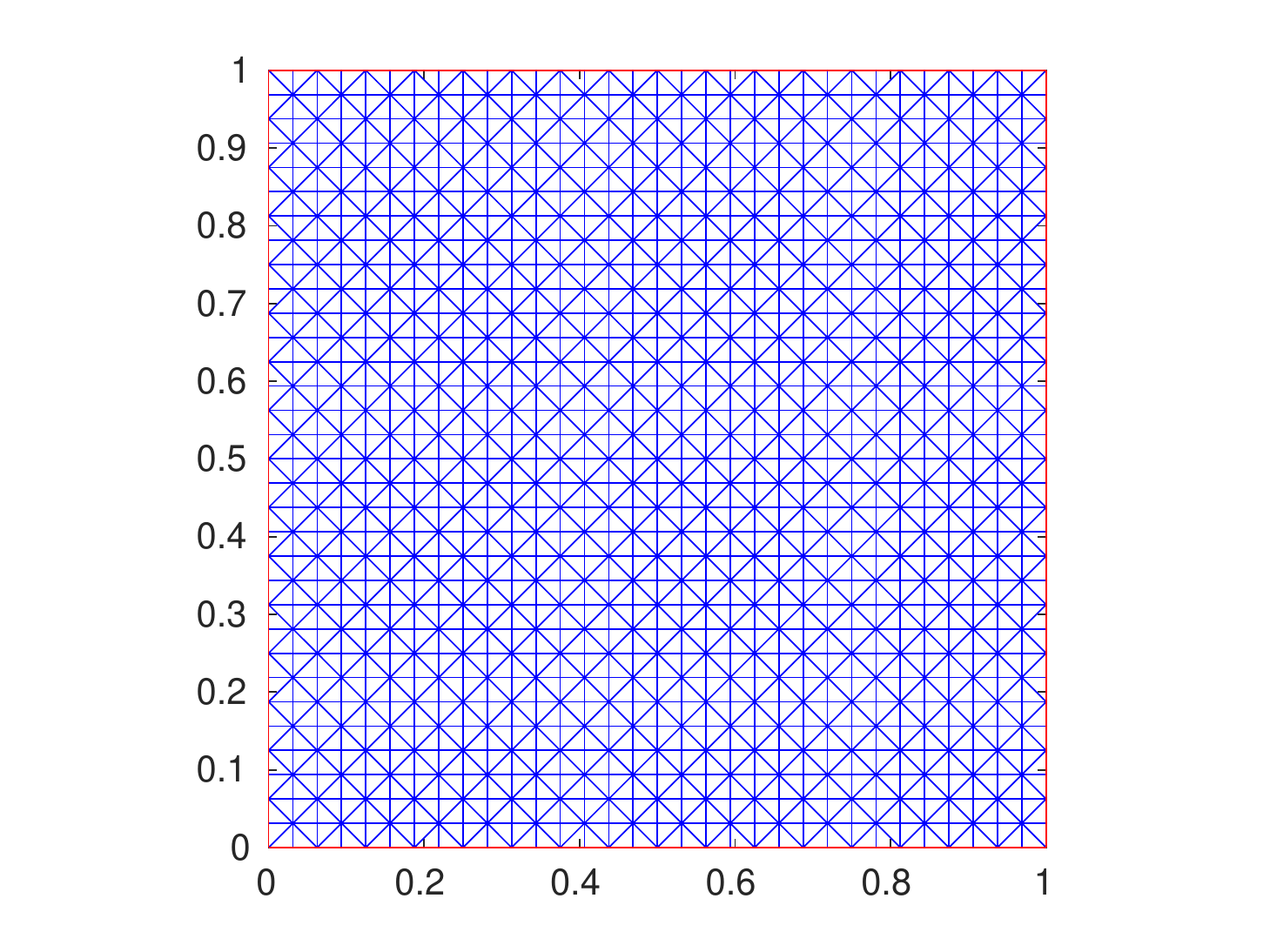}  
  \caption{2048 elements}
\end{subfigure}
\caption{Grids on the unit square  {for Problem 3}.}
\label{griduniform}
\end{figure}

\begin{table}[ht]
\caption{Convergence of the moment errors  {for Problem 3}.}
\centering
\begin{tabular}{|c|c|c|c|c|}
\hline
$N$ & $\|\sigma-\sigma_h\|$&$\|\Pi_h\sigma-\sigma_h\|$
 &$\|\sigma-K_h\sigma_h\|$
&  $\|\sigma-R_h\sigma_h\|$\\
\hline
128&3.348e-2&7.840e-3&1.285e-2&1.615e-2\\
512&1.655e-2&4.240e-3&1.011e-2&4.977e-3\\
2048&8.312e-3&1.932e-3&4.925e-3&1.289e-3\\
8192&4.161e-3&9.385e-4&2.431e-3&3.241e-4\\
32768&2.081e-3&4.656e-4&1.208e-3&8.109e-5\\
\hline
order &0.997&1.060&1.021&1.981\\
\hline
\end{tabular}
\label{ConvergenceTab}
\end{table}
\subsection{Problem 1}
Following the corner singularity of biharmonic equations analyzed in \cite{Grisvard1992}, on the L-shaped domain
\begin{equation*}
    \Omega = [-1,1]^2\backslash\big([0,1]\times[-1,0]\big),
\end{equation*} 
we consider  \eqref{mix:c} with the exact solution   
\begin{equation*} 
\begin{aligned}
u(r,\theta)&=(r^{2}\cos^{2}\theta-1)^{2}(r^{2}\sin^{2}\theta-1)^{2}r^{1+\gamma}g(\theta),\\
\gamma&=0.544483736782464,\quad\omega=\frac{3\pi}{2},\\
g(\theta)&=\left(\frac{\sin((\gamma-1)\omega)}{\gamma-1}-\frac{\sin((\gamma+1)\omega)}{\gamma+1}\right)
\big(\cos((\gamma-1)\theta)-\cos((\gamma+1)\theta)\big)\\
&-\left(\frac{\sin((\gamma-1)\theta)}{\gamma-1}-\frac{\sin((\gamma+1)\theta)}{\gamma+1}\right)
\big(\cos((\gamma-1)\omega)-\cos((\gamma+1)\omega)\big),
\end{aligned}
\end{equation*}
where $(r,\theta)$ is the polar coordinate with respect to the origin. The boundary condition is purely clamped ($\partial\Omega=\Gamma_c$). Figure~\ref{initialL} is the initial grid used in adaptive algorithms. An highly graded produced by $\eta_h$-based and $\zeta_h$-based adaptive methods are shown in Figures \ref{graded1} and \ref{graded2}, respectively.

In Figure~\ref{curveclamped}, the numerical order of convergence is evaluated using solutions after the 6th adaptive loop. It could be observed that those rates of convergence match the predicted rates in Sections \ref{sec02h} and \ref{sec01}. In addition, there is apparent superconvergence of $\|\sigma-R_h\sigma_h\|$ in Fig.~\ref{curvezeta}. The effectiveness ratio $\text{eff}=E_h/\eta_h$ and $\text{eff}=e_h/\zeta_h$ is shown in Figure \ref{eff}. Readers are referred to \cite{WuZhang2007} for a theoretical investigation of superconvergence of linear and quadratic Lagrange elements under adaptive grids. As explained in Section \ref{sec01}, the error estimator $\zeta_h$ is almost asymptotically exact.

\subsection{Problem 2}
In the second experiment, we consider \eqref{HHJ} on the L-shaped domain $\Omega$ with $\partial\Omega=\overline{\Gamma}_s\cup\overline{\Gamma}_f$, where the 
free part $\Gamma_f$ consists of two segments sharing the reentrant corner and $\Gamma_s$ is the rest part of $\partial\Omega$. The transverse load is $f=10.$ The initial grid is the same as Problem 1. We use the newest vertex bisection in adaptive algorithms and uniform quad-refinement in non-adaptive ones. For this problem, there is no explicit analytical solution and we report the convergence history of error estimators in Figure \ref{estimator}.

It is observed in Figure \ref{estimator} that the convergence under uniform refinement is rather slow.
However, the adaptive algorithms based on $\eta_h$ and $\zeta_h$  are able to recover the optimal rate of convergence with respect to $N$.

\subsection{Problem 3}
In the third experiment, we consider \eqref{mix:c} on $\Omega=[0,1]^2$ with the exact solution $$u=x_1^2(x_1-1)^2x_2^2(x_2-1)^2$$
under the purely clamped boundary condition $\partial\Omega=\Gamma_c$. We numerically compare the performance of the postprocessing scheme $R_h$ and the edge-averaging scheme $K_h$ proposed in \cite{HuMa2016,Brandts1994}. The initial grid is the $8\times8$ uniform triangulation of $\Omega$. A grid sequence is then generated by uniform newest vertex bisection, see Figure \ref{griduniform}. We note that this grid sequence is not uniformly parallel and the assumption in Lemma \ref{superclosesigma} fails.

The numerical order in Table \ref{ConvergenceTab} is evaluated using \textsf{polyfit} and the data below the second row.
It is clear from Table \ref{ConvergenceTab} that $\|\Pi_h\sigma-\sigma_h\|$ is not super-small and $\|\sigma-K_h\sigma_h\|$ is not superconvergent at all. However, it is observed that $\|\sigma-R_h\sigma_h\|$ has one order higher global superconvergence, indicating the superiority of $R_h$ in this situation.

\section*{Declarations}
\vspace{0.2cm}
\noindent\textbf{Funding} The author did not receive support from any organization for this work.

\vspace{0.2cm}
\noindent\textbf{Data Availability} Data sharing is not applicable to this article as no datasets were generated or analysed during
the current study.

\vspace{0.2cm}
\noindent\textbf{Conflicts of interest} The author has no relevant financial or non-financial interests to disclose.

\vspace{0.2cm}
\noindent\textbf{Code availability} The code used in this study is available from the author upon request.


\bibliographystyle{spmpsci}

\end{document}